\documentclass[a4paper]{amsart}

\usepackage{amsmath,amssymb,amsthm,stmaryrd}

\usepackage{color,graphicx}
\usepackage[dvipsnames]{xcolor}
\usepackage{accents}
\usepackage{enumerate}
\usepackage{varwidth}

\usepackage{fancyhdr, array,calc,graphicx,url,tabularx}
\usepackage{geometry}
\usepackage{latexsym}
\usepackage{amssymb}
\usepackage{amsmath}

\usepackage{enumerate}
\usepackage{verbatim}
\usepackage{tikz}
\usepackage[colorlinks=true,linkcolor=blue]{hyperref}

\usepackage{changepage}

\usepackage[utf8]{inputenc}
\usepackage{tikz}
\tikzset{
     block/.style={rectangle, draw, fill=red!40, text width=6em,
                   text centered, rounded corners, minimum height=3em},
     arrow/.style={-{Stealth[]}}
     }

\def\undertilde#1{\mathord{\vtop{\ialign{##\crcr
$\hfil\displaystyle{#1}\hfil$\crcr\noalign{\kern1.5pt\nointerlineskip}
$\hfil\tilde{}\hfil$\crcr\noalign{\kern1.5pt}}}}}

{
   \end{minipage}
   \vspace*{\stretch{3}}
   \clearpage
}

\def\undertilde#1{{\baselineskip=0pt\vtop
  {\hbox{$#1$}\hbox{$\scriptscriptstyle\sim$}}}{}}

\newcommand{\utilde}{\undertilde}

\renewcommand{\gg}{\gamma}

\newcommand{\bR}{{\mathbb{R}}}

\newcommand{\rest}{\restriction}
\newcommand{\la}{\langle}
\newcommand{\ra}{\rangle}

\newcommand{\card}[1]{{\vert #1 \vert} }

\renewcommand{\models}{\vDash}
\newcommand{\powerset}{{\wp}}
\newcommand{\dom}{{\rm dom}}
\newcommand{\rge}{{\rm rge}}
\newcommand{\cp}{{\rm crit }}

\newcommand{\cf}{{\rm cf}}
\newcommand{\lh}{{\rm lh}}

\newtheorem{theorem}{Theorem}[section]
\newtheorem{proposition}[theorem]{Proposition}
\newtheorem{definition}[theorem]{Definition}

\newtheorem{lemma}[theorem]{Lemma}
\newtheorem{corollary}[theorem]{Corollary}
\newtheorem{claim}[theorem]{Claim}

\newtheorem{conjecture}[theorem]{Conjecture}
\newtheorem{question}[theorem]{Question}

\newtheorem{remark}[theorem]{Remark}
\newtheorem{notation}[theorem]{Notation}

\numberwithin{figure}{section}

\newcommand{\rcl}[1]{Claim~\ref{#1}}

\newcommand{\rprop}[1]{Proposition~\ref{#1}}
\newcommand{\rthm}[1]{Theorem~\ref{#1}}
\newcommand{\rlem}[1]{Lemma~\ref{#1}}

\newcommand{\rdef}[1]{Definition~\ref{#1}}

\newcommand{\rrem}[1]{Remark~\ref{#1}}

\def\k{\kappa}
\def\a{\alpha}
\def\b{\beta}
\def\d{\delta}

\def\l{\lambda}

\def\cop{{\sf{cop}}}

\def\P{{\mathcal{P} }}
\def\W{{\mathcal{W} }}
\def\Q{{\mathcal{ Q}}}
\def\mH{{\mathcal{ H}}}
\def\K{{\mathcal{ K}}}

\def\R{{\mathcal R}}

\def\H{{\rm{HOD}}}
\def\M{{\mathcal{M}}}
\def\N{{\mathcal{N}}}

\def\T {{\mathcal{T}}}
\def\U{{\mathcal{U}}}
\def\S{{\mathcal{S}}}
\def\V{{\mathcal{V}}}
\def\X{{\mathcal{X}}}
\def\Y{{\mathcal{Y}}}
\def\Z{{\mathcal{Z}}}

\def\card#1{\left|#1\right|}

\def\iff{\mathrel{\leftrightarrow}}

\def\and{\mathrel{\kern1pt\&\kern1pt}}

\def\insegeq{\trianglelefteq}

\def\<#1>{\langle\,#1\,\rangle}

 \input xy
 \xyoption{all}

\begin{document}

\title{${\sf{AD^+}}$ implies that $\omega_1$ is a club $\Theta$-Berkeley cardinal}

\author{Douglas Blue and Grigor Sargsyan} 
\address{Douglas Blue, IMPAN, Antoniego Abrahama 18, 81-825 Sopot, Poland.}\email{dblue@impan.pl}\address{Grigor Sargsyan, IMPAN, Antoniego Abrahama 18, 81-825 Sopot, Poland. }
\email{gsargsyan@impan.pl} 

\thanks{The authors work is funded by the National Science Centre, Poland under the WeaveUNISONO call in the Weave programme, registration number UMO-2021/03/Y/ST1/00281.}

\begin{abstract}
Following \cite{bagaria2019large}, given cardinals $\k<\l$, we say $\kappa$ is a club $\l$-Berkeley cardinal if for every transitive set $N$ of size $<\l$ such that $\kappa\subseteq N$, there is a club $C\subseteq \k$ with the property that for every $\eta\in C$ there is an elementary embedding $j: N\rightarrow N$ with $\cp(j)=\eta$. We say $\kappa$ is $\nu$-club $\l$-Berkeley if $C\subseteq \k$ as above is a $\nu$-club. We say $\k$ is $\l$-Berkeley if $C$ is unbounded in $\k$. We show that under ${\sf{AD^+}}$, (1) every regular Suslin cardinal is $\omega$-club $\Theta$-Berkeley (see \rthm{main theorem}), (2) $\omega_1$ is club $\Theta$-Berkeley (see \rthm{main theorem lr} and \rthm{main theorem}), and (3) the $\utilde\d^1_{2n}$'s are $\Theta$-Berkeley---in particular, $\omega_2$ is $\Theta$-Berkeley (see \rrem{omega2}).

Along the way, we represent regular Suslin cardinals in direct limits as cutpoint cardinals (see \rthm{char extenders}). This topic has been studied in \cite{MPSC} and \cite{jackson2022suslin}, albeit from a different point of view. We also show that, assuming $V=L(\bR)+{\sf{AD}}$, $\omega_1$ is not $\Theta^+$-Berkeley, so the result stated in the title is optimal (see \rthm{lr optimal} and \rthm{thetareg optimal}).

\end{abstract}

\maketitle
\setcounter{tocdepth}{1}

\section{Introduction}
Kunen \cite{kunen1971elementary} famously showed that the existence of a nontrivial elementary embedding $j:V_{\lambda+2} \to V_{\lambda+2}$ is inconsistent with $\mathsf{ZFC}$.
The intractability of the question whether $\mathsf{ZF}$ refutes the existence of such an embedding led Woodin to define, in his set theory seminar in the 1990s, a large cardinal notion to test whether $\mathsf{ZF}$ can prove even one nontrivial inconsistency.
A cardinal $\kappa$ is a \emph{Berkeley cardinal} if for every transitive set $M$ with $\kappa\in M$, and for any ordinal $\eta<\kappa$, there is a nontrivial elementary embedding $j:M\to M$ with $\eta<\cp(j)<\kappa$.
In addition to revealing tension between axioms of infinity and the Axiom of Choice, if Berkeley cardinals are consistent with $\mathsf{ZF}$, then the Ultimate $L$ Conjecture is false \cite[Corollary 8.1]{bagaria2019large}.

We answer the determinacy version of this consistency question.
Solovay showed that assuming $\mathsf{AD}$, every subset of $\omega_1$ is constructible from a real, and hence there is a nontrivial elementary embedding from any $\mathsf{ZFC}$ model of height $\omega_1$ to itself with critical point less than $\omega_1$.
Thus $\omega_1$ is ``$\mathsf{ZFC}$-Berkeley'' for structures of height $\omega_1$.
We generalize this to transitive sets of any size less than $\Theta$ which are coded by sets of ordinals.
Moreover, we can ensure that for every club in $\omega_1$ and every such set, there is an embedding with critical point in that club, that is, adapting the terminology of \cite{bagaria2019large},

\begin{theorem}\label{Main}
	Assume ${\sf{AD^+}}$.
	Then $\omega_1$ is club $\Theta$-Berkeley.
\end{theorem}

A cardinal $\kappa$ is a $\sf{HOD}$-\emph{Berkeley} cardinal if for all transitive sets $M\in\sf{HOD}$ with $\kappa\in M$, and for every ordinal $\eta<\kappa$, there is a nontrivial elementary embedding $j:M\to M$ with $\eta<\cp(j)<\kappa$ \cite{bagaria2019large}.
In $\sf{ZFC}$, the existence of a $\sf{HOD}$-Berkeley cardinal implies the failure of the $\sf{HOD}$ Conjecture (and hence the Ultimate $L$ conjecture) \cite[Theorem 8.6]{bagaria2019large}.
It is an immediate corollary of Theorem \ref{Main} that in a $\sf{ZFC}$ forcing extension, $\omega_1$ is a $\sf{HOD}$-Berkeley cardinal for structures which are ordinal definable from a real and belong to $H_{\omega_3}$, see Corollary \ref{main corollary}.

Recall that a set of reals $A$ is $\kappa$-\emph{Suslin} if $A$ is the projection of a tree on $\omega\times\kappa$, and $\kappa$ is a \emph{Suslin cardinal} if there is a $\kappa$-Suslin set of reals which is not $\gamma$-Suslin for any $\gamma<\kappa$.
We show that every regular Suslin cardinal $\kappa$ is $\omega$-club $\Theta$-Berkeley, i.e.~for all $\omega$-clubs $C\subseteq\kappa$ and all transitive sets $M$ with $\kappa\in M$ and of size less than $\Theta$, there is a nontrivial elementary embedding $j:M\to M$ with $\cp(j)\in C$.

\begin{theorem}\label{Suslin}
	Assume $\sf{AD}^+$.
	Then every regular Suslin cardinal is $\omega$-club $\Theta$-Berkeley. Thus,  every limit Suslin cardinal is a $\Theta$-Berkeley cardinal.\footnote{This follows because if $\k$ is a singular Suslin cardinal that is a limit of Suslin cardinals then it is a limit of regular Suslin cardinals. See \cite{Jackson}.} 
\end{theorem}

Of course, $\omega_1$ is a regular Suslin cardinal.
What seems to distinguish the arguments for Theorems \ref{Main} and \ref{Suslin} is whether the club or $\omega$-club filter on the cardinal in question is an ultrafilter.
For example, we expect that if $\kappa$ is the largest Suslin cardinal, e.g.~$\utilde{\delta}^2_1$ in $L(\mathbb{R})$, then $\kappa$ is club $\Theta$-Berkeley.

Theorem \ref{Suslin} establishes the existence of \emph{limit $\omega$-club $\Theta$-Berkeley cardinals,} $\omega$-club $\Theta$-Berkeley cardinals which are limits of $\Theta$-Berkeley cardinals.

\begin{corollary}\label{limitclub}
    Assume $\mathsf{AD}^+$.
    Then every regular limit Suslin cardinal is a limit $\omega$-club $\Theta$-Berkeley cardinal.
\end{corollary}

Recall that for every $n$, the projective ordinal $\utilde\d^1_n $ is the supremum of the lengths of $\utilde{\Delta}^1_n$ prewellorderings of the reals.
The projective ordinals are analogues of $\mathsf{ZFC}$ cardinals in the setting of $\mathsf{AD}$.
We show that the even projective ordinals are $\Theta$-Berkeley.

\begin{theorem}
	Assume $\mathsf{AD}^+$.
	Then for all $n$, $\utilde\d^1_{2n}$ is $\Theta$-Berkeley.
	
\end{theorem}
\noindent
In particular, $\omega_2$ is $\Theta$-Berkeley.


A few words are in order about how these $\mathsf{AD}^+$ theorems bear on the questions whether $\mathsf{ZF}$ + ``there is a Berkeley cardinal'' or $\mathsf{ZFC}$ + ``there is a $\mathsf{HOD}$-Berkeley cardinal'' are consistent.
Consider the latter question.
Historically, large cardinals witnessed by elementary embeddings have been isolated first and subsequently shown to hold, in their measure formulations and assuming $\mathsf{AD}$, at small cardinals.
Thus $\omega_1$ is measurable, strongly compact, supercompact, and huge, and $\omega_2$ is measurable and has a significant degree of supercompactness.
Presumably, this could have happened in reverse.
Then we would need to see whether $\mathsf{ZF}$ large cardinal notions like Berkeley cardinals can ``survive'' the Axiom of Choice.
Full Berkeley cardinals cannot.
Perhaps $\mathsf{HOD}$-Berkeley cardinals do.
This paper opens the door for that eventuality.



\subsection{Acknowledgements}
The authors thank Gabriel Goldberg for fruitful discussion about choiceless cardinal phenomena under $\mathsf{AD}$ and for raising the question Theorem \ref{lr optimal} addresses.

\section{Preliminaries}

\subsection{Inner model theory}
The proofs of \rthm{main theorem lr} and \rthm{main theorem} require inner model theory. We will use the full normalization technique, and \cite[Theorem 1.4]{MPSC} in particular will play a crucial role. We will also use the HOD analysis, references for which include \cite{HODCoreModel}, \cite[Chapter 8]{OIMT}, \cite{HMMSC} and \cite{SteelCom}. We will only need the HOD analysis in models of the form $L^\Psi(\bR)$, where $\Psi$ is an iteration strategy, and the HOD analysis that we will need is the one that for a given $x\in \bR$, represents ${\sf{HOD}}_{\Psi, x}|\Theta^{L^\Psi(\bR)}$ as a $\Psi$-premouse over $x$. In this regard, the HOD analysis we need is essentially the HOD analysis of $L(\bR)$. 

The following notation will be used throughout. Suppose $\M$ is some fine structural premouse (e.g.~a hybrid premouse, hod premouse or just a pure premouse). We say that a cardinal $\k$ is a cutpoint of $\M$ if there is no extender $E\in \vec{E}^\M$ such that $\cp(E)<\k\leq \lh(E)$. By a theorem of Schlutzenberg (see \cite{zbMATH07628769}), one can remove the condition that $E\in \vec{E}^\M$. 

When we write ``$\k$ is a measurable cardinal of $\M$" or similar expressions, we mean that $\k$ is a measurable cardinal in $\M$ as witnessed by the extender sequence of $\M$. The aforementioned result of Schlutzenberg makes this convention unnecessary, but it is easier to communicate results with it.
 
Given a premouse (or any model with an extender sequence) $\M$ and an $\M$-cardinal $\nu$, we let $o^\M(\nu)=\sup(\{\lh(E): E\in \vec{E}^\M \wedge \cp(E)=\nu\})$. That is, $o^\M(\nu)$ is the Mitchell order of $\nu$. 

Following \cite[Definition 2.2]{OIMT}, for a premouse $\M$ and $E\in \vec{E}^\M$ with $\k=\cp(E)$, we let $\nu(E)=\sup((\k^+)^\M\cup \{\xi+1: \xi$ is a generator of $E\})$. We also let $\pi_E^\M$ be the ultrapower embedding given by $E$. We will often omit $\M$.

If $\M$ is a non-tame premouse such that $\M\models ``$there are no Woodin cardinals," then $\M$ has at most one $\omega_1+1$-iteration strategy, and under $\sf{AD}$, because $\omega_1$ is a measurable cardinal, $\M$ has at most one $\omega_1$-iteration strategy. For more details see \cite{OIMT}.

Suppose $\M$ is a premouse and $\Sigma$ is an iteration strategy for $\M$. If $\N$ is a normal $\Sigma$-iterate of $\M$, then we let $\T^\Sigma_{\M, \N}$ be the normal $\M$-to-$\N$ tree that is according to $\Sigma$, and if the main branch of $\T^\Sigma_{\M, \N}$ does not drop, then we let $\pi^\Sigma_{\M, \N}:\M\rightarrow \N$ be the iteration embedding given by $\T^\Sigma_{\M, \N}$. If $\N$ is a $\Sigma$-iterate of $\M$, then $\Sigma_\N$ is the iteration strategy for $\N$ given by $\Sigma_\N(\U)=\Sigma((\T^\Sigma_{\M, \N})^\frown \U)$. 

In the above situation, we say $\N$ is a complete $\Sigma$-iterate of $\M$ if $\pi_{\M, \N}$ is defined. When discussing direct limit constructions, we will use $\M_\infty(\M, \Sigma)$ for the direct limit of all complete $\Sigma$-iterates of $\M$ and $\pi_{\M, \infty}^\Sigma:\M\rightarrow \M_\infty(\M, \Sigma)$ will be the direct limit embedding. If $\N$ is a complete $\Sigma$-iterate of $\M$ then $\pi^\Sigma_{\N, \infty}:\N\rightarrow \M_\infty(\M, \Sigma)$ is the iteration embedding. 

We will often omit $\Sigma$ from the superscripts in the notation introduced above.  

\subsection{\textbf{Woodin's Derived Model Theorem}}

Assume ${\sf{ZFC-Powerset}}+``\l$ is a limit of Woodin cardinals$"+``\l^+$ exists," and suppose $g\subseteq Coll(\omega, <\l)$ is a generic.  For $\a<\l$, let $g_\a=g\cap Coll(\omega, \a)$. Set $\bR^*=\cup_{\a<\l}\bR^{V[g_\a]}$ and, working in $V(\bR^*)$, let $\Gamma^*$ be the set of those $A\subseteq \bR^*$ such that for some $\a<\l$ and for some $(T, S)\in V[g_\a]$, $V[g_\a]\models ``(T, S)$ are $<\l$-absolutely complementing" and $A=\cup_{\b\in [\a, \l)}(p[T])^{V[g_\b]}$. 

\begin{theorem}[Woodin's ${\sf{Derived\ Model\ Theorem}}$, \cite{St07DMT, St09DMT}] Assume ${\sf{ZFC-Powerset}}+``\l$ is a limit of Woodin cardinals$"+``\l^+$ exists". Suppose $g\subseteq Coll(\omega, <\l)$ is a generic. Then $L(\Gamma^*, \bR^*)\models {\sf{AD}^+}$. 
\end{theorem}

 The model $L(\Gamma^*, \bR^*)$ is the derived model of $V$ at $\l$ induced by $g$. We denote it by $D(V, \l, g)$. While $D(V, \l, g)$ is not in $V$, its theory is, and in $V$, we can refer to $D(V, \l, g)$ via the forcing language. 
 \begin{notation}\label{derived model notation} Suppose $\l$ is as above, $X\in V_\l$, $A$ is a $<\l$-uB set and $\phi$ is a formula. We write $V\models \phi^{D(\l)}[X, A]$ if whenever $g\subseteq Coll(\omega, <\l)$ is generic, $D(V, \l, g)\models \phi[X, A_g]$, where $A_g$ is the interpretation of $A$ in $V[g]$.
 \end{notation}

 Suppose $(\P, \Sigma)$ is a mouse pair of some kind and $\P$ has infinitely many Woodin cardinals. Let $\l$ be a limit of Woodin cardinals of $\P$. Then $M$ is a derived model of $(\P, \Sigma)$ at $\lambda$ if there is some genericity iteration\footnote{For more on genericity iterations, see \cite[Chapter 7]{OIMT}.} of $\P$ via $\Sigma$ with last model $\P_\omega$ such that if $\pi:\P\rightarrow \P_\omega$ is the iteration embedding, then $\pi(\l)=\omega_1$ and $M$ is the derived model of $\P_\omega$ at $\omega_1^V$ as computed by some $g\subseteq Coll(\omega, <\omega_1^V)$ which is $\P_\omega$-generic and $(\bR^*)^{\P_\omega[g]}=\bR$.

\section{$\omega_1$ is club $\Theta$-Berkeley}

Before the proof of the main theorem, we present a proof of a special but representative case.
This proof has the advantage of being more accessible while featuring most of the main ideas. 

\textbf{The main ideas.} We present the main idea behind the proofs of \rthm{main theorem lr} and \rthm{main theorem} assuming that $\M_\omega^\#$ exists. Let $\M=\M_\omega^\#$ and let $\d$ be the least Woodin cardinal of $\M$. Let $\P=\M|(\d^+)^\M$, and let $\Sigma$ be the $\omega_1+1$-iteration strategy of $\P$. It is a theorem of Woodin that if $\N$ is the direct limit of all countable iterates of $\P$ via $\Sigma$, then $\pi_{\P, \N}(\d)=\Theta^{L(\bR)}$ and the universe of $\N|\Theta^{L(\bR)}$ is just $({\sf{HOD}}|\Theta)^{L(\bR)}$. Our goal now is to generate a non-trivial embedding
\begin{center}
$j: ({\sf{HOD}}|\Theta)^{L(\bR)}\rightarrow ({\sf{HOD}}|\Theta)^{L(\bR)}$.
\end{center}
\rthm{lr optimal} shows that such an embedding cannot exist in $L(\bR)$, but we can hope to find such an embedding $j$ with the additional property that if $\gg<\Theta^{L(\bR)}$ then $j\rest ({\sf{HOD}}|\gamma)^{L(\bR)}\in L(\bR)$. We obtain such an embedding as follows. Let $E$ be the total Mitchell order 0 extender on the extender sequence of $\P$ with the property that $\cp(E)$ is the least measurable cardinal of $\P$. Set $\tau=\cp(E)$, $\P_0=\P|(\tau^{++})^\P$, $\Q_0=\pi_E(\P_0)$ and $\Q=Ult(\P, E)$. Let $\Lambda$ be the fragment of $\Sigma$ that acts on iteration trees that are above ${\sf{Ord}}\cap \P_0$, and similarly, let $\Phi$ be the fragment of $\Sigma_\Q$ that acts on iterations that are above ${\sf{Ord}}\cap \Q_0$. Let $\R$ be the direct limit of all countable $\Lambda$-iterates of $\P$ and $\S$ be the direct limit of all countable $\Phi$-iterates of $\Q$. Then $\pi_E$ generates an embedding $j^+:\R\rightarrow \S$ such that $j\rest \P_0=\pi_E\rest \P_0$. Moreover, setting $j=j^+\rest ({\sf{HOD}}|\gamma)^{L(\bR)}$, $j$ is as desired. 

\begin{theorem}\label{main theorem lr} 
Assume ${\sf{AD}} + V=L(\bR)$.
Then $\omega_1$ is a club $\Theta$-Berkeley cardinal.
\end{theorem}
\begin{proof}
Towards a contradiction, assume not. Fix a transitive $N'$ such that\\\\
(1.1) $\card{N'}<\Theta$, \\
(1.2) $\omega_1\subseteq N'$ and \\
(1.3) the set of $\a<\omega_1$ such that there is no elementary embedding $j: N'\rightarrow N'$ with the property that $\cp(j)=\a$ is stationary in $\omega_1$.\\\\
Let $\phi(u)$ be the formula expressing (1.1)-(1.3).
Thus $\phi[N']$ holds.

Fix a real $x_0$ such that $N'$ is ordinal definable from $x_0$. By minimizing the ordinal parameter, we can find $N$ such that $N$ is the ${\sf{OD}}_{x_0}$-least\footnote{With respect to a natural order on sets that are ordinal definable from $x_0$. We assume that this order $\leq_{{\sf{OD}}, x_0}$ has the following property: For $\a<\b$, $(\leq_{{\sf{OD}}, {x_0}})^{L_\a(\bR)}=(\leq_{{\sf{OD}},{x_0}})^{L_\b(\bR)}\cap L_\a(\bR)$.} $M$ such that $\phi[M]$ holds. 

We observe that $\card{N}<\utilde{\d}^2_1$.\footnote{Recall that $\utilde{\delta}^2_1$ is the supremum of $\utilde{\Delta}^2_1$ prewellorderings of $\bR$. In $L(\bR)$, $\utilde{\delta}^2_1$ can be characterized as the least $\k$ such that $L_\k(\bR)\prec_1 L(\bR)$. See \cite{KoelWoodin}.}
Indeed, because $L_{\utilde{\d}^2_1}(\bR)\prec_1 L(\bR)$, we have some $\a<\utilde{\d}^2_1$ such that $L_\a(\bR)\models {\sf{ZF-Powerset}}$ and such that for some $K\in L_\a(\bR)$, $L_\a(\bR)\models ``K$ is ${\sf{OD}}_{x_0}$ and $\phi[K]$."
Since any function $k: K\rightarrow K$ is essentially a set of ordinals, Moschovakis' Coding Lemma\footnote{See \cite{KoelWoodin} or \cite{Mo09}.} implies that $L(\bR)\models ``K$ is ${\sf{OD}}_{x_0}$ and $\phi[K]$".
Since $N$ was the ${\sf{OD}}_{x_0}$-least, it follows that $N\in L_\a(\bR)$. 

Now let $\a<\utilde{\d}^2_1$ be such that\\\\
(2.1) $L_\a(\bR)\models {\sf{ZF-Replacement}}+``N$ is the ${\sf{OD}}_{x_0}$-least $K$ such that $\phi[K]"$+$``\card{N}<\utilde{\delta}^2_1."$\\
(2.2) $L_\a(\bR)$ is the derived model of some  pair $(\P, \Sigma)$\footnote{Thus $\P$ has $\omega$ Woodin cardinals.} such that $\P$ is an $x_0$-mouse and 
\begin{center}
$({\sf{HOD}}_{x_0}|\Theta)^{L_\a(\bR)}=\M_\infty(\P, \Sigma)|\Theta^{L_\a(\bR)}$.\footnote{See \rthm{woodins thm} and also \cite{MSCR}. By writing such equalities, we mean that the two structures have the same universe.}
\end{center} 
(2.3) Letting $(\d^i_\P: i\leq \omega)$ be the Woodin cardinals of $\P$ and their limit, for some $\P$-successor cutpoint cardinal $\nu<\d^0_\P$, ${\sf{Ord}}\cap N<\pi_{\P, \infty}^\Sigma(\nu)$.\\
(2.4) For every $\P$-successor cutpoint cardinal $\nu<\d^0_\P$, $\Sigma_{\P|\nu}\in L_\a(\bR)$.\footnote{By a theorem of Steel, $\utilde{\d}^2_1$ is the least ${<}\Theta^{L_\a(\bR)}$-strong cardinal of $\M_\infty(\P, \Sigma)$. See \cite[Section 1]{NegRes} for some justifications for this condition.}\\
(2.5) $\Sigma$ has full normalization.\\\\
To obtain $(\P, \Sigma)$ as above, we first let $\a$ be the least satisfying clause (2.1) and then use \cite[Lemma 2.5]{MSCR} to build $(\P, \Sigma)$. (2.5) follows from the results of \cite{SteelCom} and \cite[Theorem 1.4]{MPSC}. Now let $\M=\M_\infty(\P, \Sigma)$. Let $\tau$ be the least measurable cardinal of $\P$ and let $E\in \vec{E}^\P$ be such that\\\\ 
(3.1) $\cp(E)=\tau$ and $E$ is total, and\\
(3.2) $\lh(E)$ is the least among all extenders of $\vec{E}^\P$ that satisfy (3.1).\\\\
Set $\Q=Ult(\P, E)$. We let $\P_0=\P|(\tau^{++})^\P$ and $\Q_0=\pi_E(\P_0)$. Notice that we can view $\P$ as a premouse over $\P_0$ and $\Q$ as a premouse over $\Q_0$. We then let $\Lambda$ be the fragment of $\Sigma$ that acts on iteration trees on $\P$ which are above  $(\tau^{++})^\P$, and let $\Phi$ be the fragment of $\Sigma_\Q$ that acts on iteration trees on $\Q$ which are above  $\pi_E((\tau^{++})^\P)$. We then have that \\\\
(4.1) $\M_\infty(\P, \Lambda)|\Theta^{L_\a(\bR)}=({\sf{HOD}}_{\P_0}|\Theta)^{L_\a(\bR)}$ and $\M_\infty(\Q, \Phi)|\Theta^{L_\a(\bR)}=({\sf{HOD}}_{\Q_0}|\Theta)^{L_\a(\bR)}$.\\\\
We now define an elementary embedding $j_{\P, \Sigma}=_{def}j:\M_\infty(\P, \Lambda)\rightarrow \M_\infty(\Q, \Phi)$ such that $j\rest \P_0=\pi_E\rest \P_0$.

Given $x\in \M_\infty(\P, \Lambda)$, fix some normal $\Lambda$-iterate $\R$ of $\P$ such that for some $y\in \R$, $\pi^{\Lambda_\R}_{\R, \infty}(y)=x$. Let $\T=\T_{\P, \R}$. Let $\U$ be the full normalization of $\T^\frown \{E\}$.\footnote{This is an iteration according to $\Sigma$.}
Clearly, $\U$ starts with $E$ and continues with the minimal $\pi_E$-copy of $\T$. Thus, $\U$ can be written as $\{E\}^\frown \W$, where $\W$ is a normal  iteration  tree on $\Q$ according to $\Phi$. If $\S$ is the last model of $\W$, then $\S=Ult(\R, E)$.
We set $j(x)=\pi_{\S, \infty}^{\Phi_\S}(\pi_E^\R(y))$. 
\begin{claim}
$j(x)$ is independent of the choice of $\R$.\footnote{Note that $j\rest \P_0=\pi_E\rest \P_0$ is immediate.}
\end{claim}
\begin{proof}
Pick another normal $\Lambda$-iterate $\R'$ of $\P$ such that for some $y'\in \R'$, $\pi_{\R', \infty}^{\Lambda_{\R'}}(y')=x$. It then follows from full normalization that we can compare $(\R, \Lambda_{\R})$ and $(\R', \Lambda_{\R'})$ via the least-extender-disagreement process and get some common iterate $(\R'', \Lambda_{\R''})$.\footnote{This comparison is entirely above ${\sf{Ord}}\cap \P_0$. See \rthm{comparing above}.} It then follows that $\pi_{\R, \R''}(y)=\pi_{\R', \R''}(y')$. Set then $y''= \pi_{\R, \R''}(y)$. 

Next, let $\T=\T_{\P, \R}$, $\T'=\T_{\P, \R'}$, $\U=\T_{\R, \R''}$ and $\U'=\T_{\R', \R''}$. Let $\R_E=Ult(\R, E)$, $\R'_E=Ult(\R', E)$ and $\R''_E=Ult(\R'', E)$. Notice that\\\\
(5.1) $\R_E$ is the last model of the full normalization of $\T^\frown \{E\}$,\\
(5.2) $\R'_E$ is the last model of the full normalization of $\T'^\frown \{E\}$,\\
(5.3) $\R''_E$ is the last model of the full normalization of $\U^\frown \{E\}$ and the full normalization of $\U'^\frown \{E\}$,\\
(5.4) $\R_E$ is a $\Phi$-iterate of $\Q$ via some normal tree $\X$ such that the full normalization of $\T^\frown \{E\}$ is $\{E\}^\frown \X$,\\
(5.5) $\R'_E$ is a $\Phi$-iterate of $\Q$ via some normal tree $\X'$ such that the full normalization of $\T'^\frown \{E\}$ is $\{E\}^\frown \X'$,\\
(5.6) letting $\Y$ and $\Y'$ be the iteration trees according to $\Phi_{\R_E}$ and $\Phi_{\R_E'}$, respectively, such that $\{E\}^\frown \Y$ is the full normalization of $\U^\frown\{E\}$ and $\{E\}^\frown \Y'$ is the full normalization of $\U'^\frown\{E\}$, then $\R''_E$ is the last model of both $\Y$ and $\Y'$, and hence, $\R''_E$ is a $\Phi$-iterate of $\Q$ via both $\X^\frown \Y$ and $\X'^\frown \Y'$.\\\\
We want to see that 
\begin{center}
(*)\ \ \ \ $\pi_{\R_E, \infty}^{\Phi_{\R_E}}(\pi_E^\R(y))=\pi_{\R'_E, \infty}^{\Phi_{\R'_E}}(\pi_E^{\R'}(y'))$.
\end{center}
It follows from (5.1)-(5.6) that (here we drop script $\Phi$ to make the formulas readable; all iteration embeddings appearing below are defined using $\Phi$)\\\\
(6.1) $\pi_{\R_E, \infty}(\pi_E^\R(y))=\pi_{\R''_E, \infty}(\pi_{\R_E,\R_E''} (\pi_E^\R(y)))$,\\
(6.2) $\pi_{\R'_E, \infty}(\pi_E^{\R'}(y))=\pi_{\R''_E, \infty}(\pi_{\R'_E,\R_E''} (\pi_E^{\R'}(y')))$,\\
(6.3) $\pi_{\R_E,\R_E''} (\pi_E^\R(y))=\pi_E^{\R''}(\pi_{\R, \R''}(y))$,\\
(6.4) $\pi_{\R'_E,\R_E''} (\pi_E^{\R'}(y'))=\pi_E^{\R''}(\pi_{\R', \R''}(y'))$,\\
(6.5) $\pi_{\R, \R''}(y)=\pi_{\R', \R''}(y')$\footnote{This clause is a consequence of the fact that both $y$ and $y'$ are preimages of $x$.}.\\\\
(*) now easily follows from (6.1)-(6.5). Indeed,
\begin{center}
\begin{align*}
\pi_{\R_E, \infty}(\pi_E^\R(y))&=\pi_{\R''_E, \infty}(\pi_{\R_E,\R_E''} (\pi_E^\R(y)))\\
&=\pi_{\R''_E, \infty}(\pi_E^{\R''}(\pi_{\R, \R''}(y)))\\
&=\pi_{\R''_E, \infty}(\pi_E^{\R''}(\pi_{\R', \R''}(y')))\\
&=\pi_{\R'_E, \infty}(\pi_E^{\R'}(y')).
\end{align*}
\end{center}
\end{proof}
The next claim essentially finishes the proof of \rthm{main theorem lr}.
\begin{claim}\label{jnn} $j(N)=N$.
\end{claim}
\begin{proof}
Working in $\P$, let $N_\P$ be the set of $x\in \P$ such that $\pi_{\P, \infty}^\Lambda(x)\in N$. $N_\P$ is definable in $\P$ by the following formula. Let $(\d^i_\P: i\leq \omega)$ be the Woodin cardinals of $\P$ and their limit. Let $\sigma[u, v, y, c]$ be a formula in the language containing $\{\in, c\}$, where $c$ is a constant for $x_0$, expressing the following:\\\\
(7.1) $v$ is a premouse and $u$ is an $\omega_1$-iteration strategy for $v$,\\
(7.2) $y\in v$,\\
(7.3) if $w$ is the ${\sf{OD}}_c$-least $w'$ such that $\phi[w']$ then $\pi_{v, \infty}^u(x)\in w$.\\\\
We then have that\\\\
(a) $x\in N_\P$ if and only if \\
(a.1) there is a $\P$-successor cutpoint cardinal $\b<\d^0_\P$ such that $x\in N_\P|\b$, and\\
(a.2) whenever $\b>(\tau^{++})^\P$ is a successor cutpoint cardinal of $\P$ such that $x\in \P|\b$, \begin{center}$\P\models (\exists \Psi \sigma[\Psi, \P|\b, x, x_0])^{D(\d^\omega_\P)}$.\end{center}
 It is important to note that the strategy $\Psi$ is just $\Lambda_{\P|\b}$, as $\P|\b$ has a unique iteration strategy. Moreover, since $\b$ is a successor cutpoint cardinal of $\P$, $\pi_{\P, \infty}^\Lambda(x)=\pi^\Psi_{\P|\b, \infty}(x)$.\footnote{For more details on such calculations see \cite[Chapter 3]{HMMSC} and \cite[Chapter 11.1]{SteelCom}.}

Now let $\psi(u, v, w)$ be the formula on the right side of the above equivalence. Then $x\in N_\P$ if and only if $\P\models \psi[x]$.

Notice that $\pi_E(N_\P)=N_\Q$, where $N_\Q$ is such that $x\in N_\Q$ if and only if $\Q\models \psi[x]$. To finish the proof of the claim, we need to show that\\\\
(*) $\pi_{\P, \infty}^\Lambda(N_\P)=N$ and $\pi_{\Q, \infty}^\Phi(N_\Q)=N$. \\\\
We only establish the first equality as the second is very similar.\footnote{This is once again a standard calculation in the theory of hod computations, and it goes back to Woodin's computation of HOD of $L(\bR)$. See \cite[Chapter 8]{OIMT} or \cite[Chapter 4]{HMMSC}.} 

Suppose $x\in \pi_{\P,\infty}^\Lambda(N_\P)$. We want to see that $x\in N$. Let $\S$ be a $\Lambda$-iterate of $\P$ such that $x=\pi_{\S, \infty}^{\Lambda_\S}(y)$ for some $y\in \S$. We then have that $\S\models \psi[y]$. Since we can realize $L_\a(\bR)$ as the derived model of $\S$, we have that $\pi_{\S, \infty}^{\Lambda_\S}(y)\in N$.

Conversely, suppose $x\in N$. Let $(y, \S)$ be such that $\S$ is a $\Lambda$-iterate of $\P$, $y\in \S$, and $\pi_{\S, \infty}^{\Lambda_\S}(y)=x$. Then $\S\models \psi[y]$, which implies that $y\in \pi_{\P, \S}(N_\P)$. Therefore, $x=\pi_{\S, \infty}^{\Lambda_\S}(y)\in \pi_{\P, \infty}^\Lambda(N_\P)$. 
\end{proof}
To finish the proof of \rthm{main theorem lr}, we need to produce a club $C\subseteq \omega_1$ such that for each $\a\in C$, there is an embedding $k: N\rightarrow N$ with $\cp(k)=\a$. Above we have produced an elementary embedding $j_{\P, \Sigma}:N\rightarrow N$ such that $\cp(j_{\P, \Sigma})=\tau$, for $\tau$ the least measurable cardinal of $\P$. Let $(\P_\a: \a<\omega_1)$ be the sequence of linear iterates of $\P$ by $E$ and its images, and let $\tau_\a$ be the least measurable cardinal of $\P_\a$. Then $\cp(j_{\P_\a, \Sigma_{\P_\a}})=\tau_\a$, and since $C=\{\tau_\a: \a<\omega_1\}$ is a club, we get a contradiction to the fact that $\phi[N]$ is true. 
\end{proof}

\section{Good and very-good pointclasses, a review of \cite{ADRUB}.}
We review concepts from coarse descriptive inner model theory used in the proof of \rthm{main theorem}.
Many of the concepts have appeared in \cite{ADRUB} and elsewhere, and many of them are due to Woodin. A reader familiar with them can skip this section and return to it as needed. 


\subsection{Very good pointclasses}
Let $\bR$ be the Baire space.
Following \cite[Chapter 3]{MouseSetWoodin}), we say that $\Gamma$ is a \textit{good pointclass} if
\begin{enumerate}
\item $\Gamma$ is closed under recursive substitution and number quantification,
\item $\Gamma$ is $\omega$-parametrized,\footnote{So there is a set $U^k\subseteq \omega\times \bR^k$ such that $U^k\in \Gamma$, and for every set $A\in \Gamma$, if $A\subseteq \bR^k$, then there is an integer $n$ such that $x\in A\iff (n, x)\in U^k$.}
\item $\Gamma$ has the scale property, and
\item $\Gamma$ is closed under $\exists^\bR$.
\end{enumerate}

Each good pointclass has its associated $C_\Gamma$ operator. For $x\in \bR$, 
\begin{center}
$C_\Gamma(x)=\{ y\in \bR: y$ is $\Gamma$-definable from $x$ and a countable ordinal$\}$.\footnote{This definition is usually used under ${\sf{AD}}$.}
\end{center}
The $C_\Gamma$ operator can be extended to sets in ${\sf{HC}}$ via the category quantifier.\footnote{This is done, for instance, in \cite[Chapter 3]{MouseSetWoodin}.}

 Let $C_\Gamma^\a$ denote the $\a$th iterate of $C_\Gamma$ so that, e.g., $C^2_\Gamma(a)=C_\Gamma(C_\Gamma(a))$. We only need this notion for $\a\leq \omega$. Set $C^\omega_\Gamma(a)=\cup_{n<\omega}C^n_{\Gamma}(a)$. 
 
 Suppose $T$ is the tree of a $\Gamma$-scale. For each $\a<\omega_1$, we let $\k_\a$ be the $\a$th-infinite cardinal of $L[T,a]$. We can then simply set $C^\a_\Gamma(a)=H_{\k_\a}^{L[T, a]}$. Then, using this definition, we have $C_\Gamma(a)=C^1_\Gamma(a)$.
 
Given a transitive $P\models {\sf{ZFC-Replacement}}$, we say $P$ is a \textit{$\Gamma$-Woodin} if for some $\d$,
\begin{enumerate}
\item $P\models ``\d$ is the only Woodin cardinal".
\item $P=C^\omega_\Gamma(P)$,
\item for every $P$-cardinal $\eta<\d$, $C_\Gamma(V^P_\eta)\models ``\eta$ is not a Woodin cardinal".
\end{enumerate}
We let $\d^P$ be the Woodin cardinal of $P$.

A sequence $(A_n: n<\omega)\subseteq \bR^{\omega}$ is a \emph{self-justifying-system} (sjs) if  for each $n\in \omega$, 
\begin{enumerate}
\item there is a sequence $(A_{m_k}: k\in \omega)$ that codes a scale on $A_n$, and
\item there is $m<\omega$ such that $A_n^c=A_m$.
\end{enumerate} 

Let $T_0$ be the theory 
\begin{enumerate}
\item ${\sf{AD^+}}+{\sf{ZF}}-{\sf{Powerset\ Axiom}}$,
\item $``\Theta$ exists,"\footnote{More precisely, ``there is an ordinal which is not the surjective image of $\bR$".} and
\item $V=L_{\Theta^{+}}(C, \bR)$ for some $C\subseteq \bR$.
\end{enumerate}
\begin{definition}\label{vg pointclass} Suppose $\Gamma$ is a good pointclass. Then $\Gamma$ is a \textbf{very good pointclass (vg-pointclass)} if there is a sjs $\vec{A}=(A_n: n\in \omega)$, $\gg<\Theta^{L(\vec{A}, \bR)}$, a $\Sigma_1$-formula $\phi$, and a real $x$ such that $L_\gg(\vec{A}, \bR)$ is the least initial segment of $L(\vec{A}, \bR)$ that satisfies $T_0+\phi(x)$ and $\Gamma=(\Sigma^2_1(\vec{A}))^{L_{\gg}(\vec{A}, \bR))}$. We say $M_\Gamma=_{def}L_\gg(\vec{A}, \bR)$ is the parent of $\Gamma$.
\end{definition}

If $\Gamma$ is a vg-pointclass and $M_\Gamma=L_\gg(\vec{A}, \bR)$ is its parent, then for any countable transitive $a$, $C_\Gamma(a)={\sf{OD}}^{M_\Gamma}(\vec{A}, a)$. 


\subsection{$\Gamma$-excellent pairs}
 Suppose $\Gamma$ is a vg-pointclass. We say that $\vec{B}\subseteq \bR^{\omega}$ is a \textit{weakly} $\Gamma$-\textit{condensing sequence} if 
\begin{enumerate}
\item $B_0$ codes a sjs such that $M_\Gamma=L_\gg(B_0, \bR)$ and $\Gamma=(\Sigma^2_1(B_0))^{L_\gg(B_0, \bR)}$,
\item $B_1=\{ (x, y)\in\bR^2: y\in C_\Gamma(x)\}$, 
\item $B_2=B_1^c$,
\item $B_3$ is any ${\sf{OD}}^{M_\Gamma}(B_0)$ set,\footnote{We will need the freedom to include any ${\sf{OD}}^{M_\Gamma}(B_0)$ set of reals into our condensing sequence.}
\item $(B_{2i+1}: i\in [2, \omega))\subseteq \Gamma$ is a scale on $B_1$,
 \item $(B_{2i}: i\in [2, \omega))\subseteq M_\Gamma$ is a scale on $B_2$, and
 \item for every $i\in [2, \omega)$, $M_\Gamma\models ``B_{2i}$ is ordinal definable from $B_0."$
\end{enumerate}

Suppose $\Gamma$ is a vg-pointclass and $M_\Gamma=L_\gg(\vec{A}, \bR)$ is its parent. Suppose $B\in M_\Gamma\cap \powerset(\bR)$ is ${\sf{OD}}^{M_\Gamma}(\vec{A})$, and suppose $a\in HC$ is a transitive set. Consider the term relation $\tau^a_B$ consisting of pairs $(p, \sigma)$ such that 
\begin{enumerate}
\item $p\in Coll(\omega, a)$,
\item $\sigma\in C_\Gamma(a)$ is a standard $Coll(\omega, a)$-name for a real, and
\item for co-meager many $g\subseteq Coll(\omega, a)$ (in the relevant topology) such that $p\in g$, $\sigma(g)\in B$.\footnote{Here and elsewhere, $\sigma(g)$ is the realization of $\sigma$ by $g$.}
\end{enumerate}
Then because $\tau^a_B$ is ${\sf{OD}}^{M_\Gamma}(\vec{A}, a)$, we have that $\tau^a_B\in C_\Gamma(C_\Gamma(a))$. 
 Given $k\in \omega$, we let $\tau^a_{B, k}=\tau^{C^{k}_\Gamma}_{B, 0}$. Thus, for every $k\in \omega$, $\tau^a_{B, k}\in C^{k+2}_\Gamma(a)$.

We say $\vec{B}$ is a $\Gamma$-\textit{condensing sequence} if it is a weakly $\Gamma$-condensing sequence with the additional property that for any transitive sets $a, b, M\in HC$ such that
\begin{enumerate}
\item $a\in M$ and
\item there is an embedding $\pi: M\rightarrow_{\Sigma_1} C^\omega_\Gamma(b)$ such that $\pi(a)=b$ and for every $i, k\in \omega$, $\tau^b_{B_i, k}\in rng(\pi)$,
\end{enumerate}
$M=C^\omega(a)$ and for any $i, k\in \omega$, $\pi^{-1}(\tau^b_{B_i, k})=\tau^a_{B_i, k}$.
 If $P$ is a $\Gamma$-Woodin and $B\in {\sf{OD}}^{M_\Gamma}(\vec{A})$, then for $k\in \omega$, we let $\tau^P_{B, k}=\tau^{V_{\d^P}^P}_{B, k}$. 
\begin{definition}[{\cite[Definition 1.6]{ADRUB}}] Suppose $\Gamma$ is a vg-pointclass and $M_\Gamma=L_\gg(\vec{A}, \bR)$ is its parent. Suppose $P$ is a $\Gamma$-Woodin and $\Sigma$ is an $\omega_1$-iteration strategy for $P$. Suppose $B\in M_\Gamma\cap \powerset(\bR)$ is ${\sf{OD}}^{M_\Gamma}(\vec{A})$. 
\begin{enumerate}
\item We say $\Sigma$ is a $\Gamma$-fullness preserving strategy for $P$ if whenever $i:P\rightarrow Q$ is an iteration of $P$ via $\Sigma$, $Q$ is a $\Gamma$-Woodin. 
\item Given that $\Sigma$ is $\Gamma$-fullness preserving, we say $\Sigma$ respects $B$ if whenever $i:P\rightarrow Q$ is an iteration of $P$ via $\Sigma$, for every $k$, $i(\tau^P_{B, k})=\tau^Q_{B, k}$.  
\end{enumerate}
\end{definition} 

The following theorem, which probably is originally due to Woodin, is unpublished. For its proof, see the discussion after \cite[Theorem 1.7]{ADRUB}.
\begin{theorem}\label{coarse capturing} Assume ${\sf{AD^+}}$ and suppose $\Gamma$ is a vg-pointclass. Let $M_\Gamma=L_\gg(\vec{A}, \bR)$\footnote{Here $\vec{A}$ is any sjs for which $M_\Gamma=L_\gg(\vec{A}, \bR)$.} be its parent and $A\in {\sf{OD}}(\vec{A})^{M_\Gamma}$.  Then there is a pair $(R, \Psi)$ and a $\Gamma$-condensing sequence $\vec{B}$ such that
\begin{enumerate}
\item $R$ is a $\Gamma$-Woodin,
\item $\Psi$ is a $\Gamma$-fullness preserving $\omega_1$-iteration strategy for $P$,
\item for each $i$, $\Psi$ respects $B_i$,
\item for every $\Psi$-iterate $Q$ of $R$, for every $i\in \omega$ and for every $Q$-generic $g\subseteq Coll(\omega, \d^Q)$, $\tau^Q_i(g)=Q[g]\cap B_i$,
\item for any iteration tree $\T\in \dom(\Psi)$, $\Sigma(\T)=b$ if and only if either
\begin{enumerate}
\item $C_\Gamma(\M(\T))\models ``\d(\T)$ is not a Woodin cardinal" and $b$ is the unique well-founded cofinal branch $c$ of $\T$ such that $C_\Gamma(\M(\T))\in \M^\T_c$, or
\item $C_\Gamma(\M(\T))\models ``\d(\T)$ is a Woodin cardinal" and $b$ is the unique well-founded cofinal branch $c$ of $\T$ such that letting $Q=C^\omega_\Gamma(\M(\T))$, $\M^\T_c=Q$, and for every $i\in \omega$, $\pi^\T_c(\tau^R_{B_i})=\tau^Q_{B_i}$.
\end{enumerate}
\item $\Psi$ respects $A$.
\end{enumerate}
Moreover, for any set $a\in {\sf{HC}}$, there is $(R, \Psi)$ as above such that $a\in R$. 
\end{theorem}

\begin{definition}\label{excellent pair} Suppose $\Gamma$ is a vg-pointclass. Then $(R, \Psi)$ is a $\Gamma$-excellent pair if for some $\Gamma$-condensing sequence $\vec{B}$, $(R, \Psi)$ has properties 1-5 described in \rthm{coarse capturing} as witnessed by $\vec{B}$.
\end{definition}

\subsection{Reflection points}

 Suppose $(P, \Psi)$ is a $\Gamma$-excellent pair. For $n\leq \omega$, we let $\M_n^{\Psi, \#}$ be the minimal active $\Psi$-mouse that has exactly $n$ Woodin cardinals. Under ${\sf{AD}}$, $\M_n^{\Psi, \#}$, as a $\Psi$-mouse, has a unique $\omega_1$-iteration strategy. Letting $\Psi^+_n$ be this iteration strategy, we have that ${\sf{Code}}(\Psi)$ is projective in ${\sf{Code}}(\Psi^+_n)$.\footnote{${\sf{Code}}:(\cup_{n<\omega} {\sf{HC}}^n)\rightarrow \bR$ is a function that codes subsets of ${\sf{HC}}^n$. As is argued in \cite{ADRUB}, $\Psi$ can be interpreted in generic extensions of $\M_1^{\Psi, \#}$ and its iterates via $\Psi^+_n$, which easily implies the claim that ${\sf{Code}}(\Psi)$ is projective in ${\sf{Code}}(\Psi^+_n)$.} 

 Suppose $(P, \Psi)$ is a $\Gamma$-excellent pair. Then $L^\Psi(\bR)$ is the minimal $\Psi$-mouse containing all the ordinals and reals. It can be defined as in \cite{MSCR} and in \cite[Chapter 3.7]{HMMSC}. Because $\bR$ is not well-ordered, the above references build $L^\Psi(\bR)$ relative to $\Psi^+_2$, though in the case of excellent pairs, the same construction would work relative to $\Psi$.\footnote{\cite{SchlT} is the general reference for hybrid mice over the reals.}

Suppose $(P, \Psi)$ is a $\Gamma$-excellent pair. Then $\utilde{\delta}^2_1(\Psi)$ is the least ordinal $\a$ such that $L_\a^\Psi(\bR)\prec^{\bR}_{1} L^\Psi(\bR)$. Here $\prec_{n}^X$ stands for elementarity with respect to $\Sigma_n$-formulas with parameters from $X$. When discussing $L^\Psi(\bR)$, we will omit the superscript $\bR$, as it is part of the language of $L^\Psi(\bR)$ (see \cite[Chapter 2.4]{KoelWoodin}).

\textbf{$\Sigma_1$-reflection for $L^\Psi(\bR)$.}
Suppose $(P, \Psi)$ is a $\Gamma$-excellent pair, $\phi$ is a formula, and $x$ is a real. Then $\b<\utilde{\delta}^2_1(\Psi)$ is a $(T_0, \phi, x)$-\emph{reflection point} if 
\begin{enumerate}
\item $L^\Psi_\b(\bR)\models T_0$,
\item $L^\Psi_\b(\bR)\models \phi[x]$, and
\item $\powerset(\bR)\cap (L^\Psi_{\b+2}(\bR)-L^\Psi_{\b+1}(\bR))\not =\emptyset$.
\end{enumerate} 
For each $(\phi, x)$ such that $L(\bR)\models \phi[x]$, the set of $(T_0, \phi, x)$-reflection points is unbounded below $\utilde{\delta}^2_1(\Psi)$ (see \cite[Chapter 2.4]{KoelWoodin}, \cite{ScalesLR} and \cite{SchlT}).

 Suppose $(P, \Psi)$ is a $\Gamma$-excellent pair. Modifying the terminology of \cite{ScalesLR}, we say $\b$ ends a $(T_0, \Psi)$-gap if clause (1) and (3) above hold. If $\b$ ends a $(T_0, \Psi)$-gap, then we set $\Theta^\b=\Theta^{M_\b}$, $M_\b=L_\b^\Psi(\bR)$, and $\mH_{Y, x}^{\b}=(\H(Y, x))^{M_\b}$.

 It is shown in \cite{ADPlusBook} that for each $\b$ that is a $(T_0, \phi, x)$-reflection point, for any set $Y\in M_\b$, and for any real $x$, $\mH_{\Psi, Y, x}^{\b}=\mH_{\Psi, Y}^\b[x]$.\footnote{One could also simply make this condition part of being a $(T_0, \Psi)$-gap.}

\subsection{Coarse tuples}

 The following definition is essential for the arguments to come.

\begin{definition}\label{capturing a cardinal} Suppose $\nu<\Theta$. Then $(R, \Psi, H, \a)$ is a \textbf{coarse tuple} if
the following conditions hold:
\begin{enumerate}
\item For some very-good pointclass $\Gamma$, $(R, \Psi)$ is a $\Gamma$-excellent pair. 
\item  $H: \bR\rightarrow V$ is a partial function such that $\dom(H)\subseteq\{x\in \bR: R\in L_1[x]\}$.
\item For every $x\in \dom(H)$, setting $H(x)=(\P_x, \Sigma_x)$, $(\P_x, \Sigma_x)$ is a $\Psi$-pure mouse pair over $x$.
\item For every $x$, $\P_x\models {\sf{ZFC}}$ and has exactly $\omega$-Woodin cardinals.
\item $\P$ is $\omega$-small with respect to $\Psi$.\footnote{I.e., there is no $\gg<{\sf{Ord}}\cap \P$ such that $\P|\gg$ is active and has infinitely many Woodin cardinals.}
\item $\a$ is a $(T_0, \Psi, \phi)$-reflection point and for each $x\in \dom(H)$, the derived model of $(\P_x, \Sigma_x)$ is $L_\a^\Psi(\bR)$.\footnote{The meaning of this statement is as follows: For any generic $g\subseteq Coll(\omega, \bR)$, letting $\Q$ be the $\Sigma_x$-iterate of $\P_x$ obtained via the $(g(i): i<\omega)$-genericity iteration of $\P_x$, the derived model of $\Q$ as computed in $\Q(\bR)$ is $L_\a^\Psi(\bR)$. In particular, the derived model calculations are not only independent of the genericity iterations but also the real $x$ used in this calculations. See \cite{DMATM}.}
\item $\M_\infty(\P_x, \Sigma_x)|\Theta^{\a}=\mH_x^\a|\Theta^\a$\footnote{Here we mean that the universe of the $\Psi$-premouse $\M_\infty(\P_x, \Sigma_x)|\Theta^{\a}$ is  $\mH_x^\a|\Theta^\a$.} and $\Theta^\a=\pi_{\P_x, \infty}^{\Sigma_x}(\d)$, where $\d$ is the least Woodin cardinal of $\P_x$. 
\item For any $x\in \dom(H)$ and any $\P$-successor cutpoint cardinal $\b<\d$, where $\d$ is the least Woodin cardinal of $\P_x$, $(\Sigma_x)_{\P_x|\b}\in L_\a^\Psi(\bR)$.
\end{enumerate}
We say that $(R, \Psi, H, \a)$ \textbf{absorbs} $\nu$ if $(R, \Psi, H, \a)$ is a coarse tuple such that  $\nu<(\utilde{\delta}^2_1)^{L_\a^\Psi(\bR)}$.\footnote{The results of \cite[Chapter 8]{OIMT} relativized to $\Psi$ imply that for any $x\in \dom(H)$, $(\utilde{\delta}^2_1)^{L_\a^\Psi(\bR)}$ is the least $<\Theta^\a$-strong cardinal of $\M_\infty(\P_x, \Sigma_x)$.}
\end{definition}

\begin{remark}\label{lr} Assuming $V=L(\bR)$, we could just work with ordinary pure mouse pairs. In this case, $H(x)$ is simply the least initial segment of $\M^\#_\omega(x)$ that has the desired properties. 
\end{remark}

\begin{theorem}\label{woodins thm} Assume ${\sf{AD^+}}$ and suppose $(R, \Psi)$ is a $\Gamma$-excellent pair for some $\Gamma$. Suppose $\nu$ is less than the largest Suslin cardinal of $L^\Psi(\bR)$. Then there is a coarse tuple $(R, \Psi, H, \a)$ absorbing $\nu$. 
\end{theorem}
\rthm{woodins thm} can be demonstrated by combining \rthm{coarse capturing}, the hod analysis of $L(\bR)$ relativized to $L^\Psi(\bR)$,\footnote{This is needed to get clause (7) of \rdef{capturing a cardinal}.} and  the results of \cite{MSCR}.\footnote{This is needed to get clause (6) of \rdef{capturing a cardinal}. In particular, see the proof of \cite[Lemma 2.5]{MSCR}. Clause (8) of \rdef{capturing a cardinal} can be established using the arguments of \cite[Section 1]{NegRes}.} 


\begin{remark}\label{changing psi to fine structural one} Suppose $(R, \Psi, H, \a)$ witnesses \rthm{woodins thm}. Let $x\in \dom(H)$, and let $\d$ be the least Woodin cardinal of $\P_x$. Set $\P=\P_x|(\d^{+\omega})^{\P_x}$ and $\Sigma=\Sigma_{\P_x}$. Then it can be shown that some complete $\Sigma$-iterate of $\Q$ of $\P_x$ is such that $(\Q, \Sigma)$ is $\Gamma$-excellent, where letting $\a$ witness clause (6) of \rdef{capturing a cardinal}, $\Gamma=(\Sigma^2_1({\sf{Code}}(\Psi), x))^{L_\a^\Psi(\bR)}$. This is because, by the results of \cite{ScalesLR} and \cite{SchlT}, there is a weakly $\Gamma$-condensing sequence $\vec{A}=(A_i: i<\omega)$ such that for each $i$, $A_i\in {\sf{OD}}_{\Psi, x}^{L_\a^\Psi(\bR)}$. $\Q$ is then a $\Sigma$-iterate of $\P$ such that $\Sigma_\Q$ respects each $A_i$. We can find such a $\Q$ using standard arguments from the hod analysis. See \cite{HODCoreModel}, \cite[Chapter 8]{OIMT} and \cite{HMMSC}.
\end{remark}

\section{Cutpoint Suslin cardinals on a cone}

In order to prove that every regular Suslin cardinal is $\omega$-club $\Theta$-Berkeley, we need to represent Suslin cardinals as \textit{cutpoint} cardinals in various ${\sf{HOD}}$-like models. This topic has been studied in \cite{MPSC} and \cite{jackson2022suslin}. The present method is motivated by a precursor of \cite[Theorem 0.3]{jackson2022suslin}. 

\begin{theorem}\label{char extenders} Assume ${\sf{AD^+}}$ and that $\d$ is a regular Suslin cardinal such that there is a triple $(R, \Psi, H, \a)$ absorbing $\d$. Then letting $H(x)=(\P_x, \Sigma_x)$, for a Turing cone of $x$, $\d$ is a limit of cutpoint cardinals in $\M_\infty(\P_x, \Sigma_x)$, and hence $\delta$ is also a cutpoint cardinal.
\end{theorem}
\begin{proof} 
For each $x\in \dom(H)$, set $\M_x=\M_\infty(\P_x, \Sigma_x)$. We assume towards a contradiction that\\\\
(*) for a Turing cone of $x$, there is $\k<\d$ such that $o^{\M_x}(\k)\geq \d$. \\\\
Because $\d$ is a regular cardinal, we have that for every $x\in \bR$, $\M_x\models ``\d$ is a  measurable cardinal." To see this, assume not. Let $\Q$ be a $\Sigma_x$-iterate of $\P_x$ such that $\d\in \rge(\pi_{\Q, \infty})$, and set $\d_\Q=\pi_{\Q, \infty}^{-1}(\d)$. We then have that $\Q\models ``\d_\Q$ is a regular non-measurable cardinal." But then $\pi_{\Q, \infty}[\d_\Q]$ is cofinal in $\d$, implying that $\cf(\d)=\omega$.  It then follows from (*) that\\\\
(**) for a Turing cone of $x$, there is $\k<\d$ such that $o^{\M_x}(\k)\geq (\d^+)^{\M_x}$. \\\\
Because $\d$ is a Suslin cardinal, we have a tree $T$ on $\omega\times \delta$ such that $p[T]$ is not $\a$-Suslin for any $\a<\d$. Because $(R, \Psi, H, \a)$ absorbs $\d$, we have that for a Turing cone of $x$, $T\in \M_x$. Thus, we have an $x\in \dom(H)$ such that\\\\
(***) there is $\k<\d$ such that $o^{\M_x}(\k)\geq (\d^+)^{\M_x}$ and $T\in \M_x$. \\\\
Let $(\k, \iota)$ be the lexicographically least pair $(\nu, \zeta)$ such that $\nu$ witnesses (***) and, letting $F=\vec{E}^{\M_x}(\zeta)$, $\cp(F)=\nu$ and $T\in Ult(\M_x, F)$.  Thus, $\k$ is a limit of cutpoints of $\M_x$. Let $\Q$ be a $\Sigma_x$-iterate of $\P_x$ such that $(\k, \d, T, E)\in \rge(\pi_{\Q, \infty})$ where $E=\vec{E}^{\M_x}(\iota)$. Set $\Lambda=(\Sigma_x)_\Q$. 
Given a complete $\Lambda$-iterate $\R$ of $\Q$, let $s_\R=_{def}(\k_\R, \d_\R, T_\R, E_\R)\in \R$ be such that 
\begin{center}
$\pi_{\R, \infty}(s_\R)=(\k, \d, T, E)$. 
\end{center}
If $\R$ is a complete $\Lambda$-iterate of $\Q$, let $\R_E=Ult(\R, E_\R)$. Let $(f_\Q, s_\Q)\in \Q$ be such that $s_\Q\in \nu(E_\Q)^{<\omega}$, $f_\Q:[\kappa_\Q]^{\card{s_\Q}}\rightarrow \Q|\k_\Q$, and $\pi_{E_\Q}(f_\Q)(s_\Q)=T_\Q$. Thus, if $\R$ is a complete $\Lambda$-iterate of $\Q$, then $\pi_{E_\R}(f_\R)(s_\R)=T_\R$.

Say $(\l, s)\in \k\times \k^{<\omega}$ is \textit{good} if there is a complete $\Lambda$-iterate $\R$ of $\Q$ such that $(\l, s)=\pi_{\R_E, \infty}(\d_\R, s_\R)$. Suppose $(\l, s)$ is good and $\R$ witnesses it. Then let 
\begin{center}
$T_{\R, \l, s}=\pi_{\R_E, \infty}(T_\R)$.
\end{center}
\begin{lemma} Suppose $(\l, s)$ is good as witnessed by $\R$ and $\R'$. Then $T_{\R, \l, s}=T_{\R', \l, s}$.
\end{lemma}
\begin{proof} Let $\T=\T_{\Q, \R}$ and $\T'=\T_{\Q, \R'}$. Let $\zeta$ be the $\R_E$-successor of $o^{\R_E}(\k_\R)$ and $\zeta'$ be the $\R'$-successor of $o^{\R'_E}(\k_{\R'})$. Let $(\R'', \Phi)$ be a common iterate of $(\R_E|\zeta, \Lambda_{\R_E})$ and $(\R'_E|\zeta', \Lambda_{\R'_E})$. It is enough to show that $\pi_{\R_E|\zeta, \R''}(T_\R)=\pi_{\R'_E|\zeta', \R''}(T_{\R'})$. 

Let now $\Y$ be the result of copying $\T_{\R_E|\zeta, \R''}$ onto $\R_E$ via $id$, $\Y'$ be the result of copying $\T_{\R'_E|\zeta', \R''}$ onto $\R_E'$ via $id$, $\X=\T^\frown \{E_\R\}^\frown \Y$, and $\X'=\T'^\frown\{E_{\R'}\}^\frown \Y'$. Notice that we have\\\\
(1.1) $\pi^\Y\rest \R_E|\zeta=\pi^{\T_{\R_E|\zeta, \R''}}$ and $\pi^{\Y'}\rest \R'_E|\zeta'=\pi^{\T_{\R'_E|\zeta', \R''}}$,\\
(1.2) $\pi_{\R_E|\zeta, \R''}(\k_\R, \d_\R, s_\R)=\pi_{\R'_E|\zeta, \R''}(\k_{\R'}, \d_{\R'}, s_{\R'})=_{def}(\k_0, \d_0, s_0)$,\\
(1.3) $\pi_{\R_E|\zeta, \R''}(T_\R)=\pi^\X(f_\Q)(\pi_{\R_E|\zeta, \R''}(s_\R))=\pi^\X(f_\Q)(s_0)$,\\
(1.4) $\pi_{\R'_E|\zeta', \R''}(T_{\R'})=\pi^{\X'}(f_\Q)(\pi_{\R'_E|\zeta', \R''}(s_{\R'}))=\pi^{\X'}(f_\Q)(s_0)$.\\\\
Let $\U$ be the full normalization of $\X$ and $\U'$ be the full normalization of $\X'$. Notice that if $\S$ is the last model of $\U$, and $\S'$ is the last model of $\U'$, then\\\\
(2.1) $\R''\insegeq \S$ and $\R''\insegeq \S'$,\\
(2.2) $(\Lambda_{\S})_{\R''}=(\Lambda_{\S'})_{\R''}$,\footnote{This follows from full normalization of $\Lambda$.} and\\
(2.3) the least-extender-disagreement comparison of $(\S, \Lambda_\S)$ and $(\S', \Lambda_{\S'})$ is above ${\sf{Ord}}\cap \R''$.\\\\
Let then $(\W, \Lambda_\W)$ be the common iterate of $(\S, \Lambda_\S)$ and $(\S', \Lambda_{\S'})$ obtained via the least-extender-disagreement comparison process. We then have that\\\\
(3.1) $\pi_{\R_E|\zeta, \R''}(T_\R)=\pi_{\Q, \W}(f_\Q)(s_0)$,\\
(3.2) $\pi_{\R'_E|\zeta', \R''}(T_{\R'})=\pi_{\Q, \W}(f_\Q)(s_0)$.\\\\
It then follows from (3.1)-(3.2) that $\pi_{\R_E, \R''}(T_\R)=\pi_{\R', \R''}(T_{\R'})$.
\end{proof}
Thus, $T_{\R, \l, s}$ is independent of $\R$, and so we let $T_{\l, s}=T_{\R, \l, s}$. 

\begin{lemma}\label{p1} For each good $(\l, s)$, $p[T_{\l, s}]\subseteq p[T]$. 
\end{lemma}
\begin{proof} Suppose $(y, h)\in p[T_{\l, s}]$. Let $\R$ witness that $(\l, s)$ is good. Let $\S$ be a complete $\Lambda_{\R_E}$-iterate of $\R_E$ such that $h\subseteq \rge(\pi_{\S, \infty})$. We can further assume that $\T=_{def}\T_{\R_E, \S}$ is based on $\R|\zeta$, where $\zeta$ is the $\R_E$-successor of $o^{\R_E}(\k_\R)$. Let $(\a_n: n\in \omega)$ be such that $h(n)=\pi_{\S, \infty}(\a_n)$. Let $U=\pi_{\R_E, \S}(T_\R)$. Thus if $h'(n)=\a_n$, then $(y, h')\in [U]$. 

Notice that $\{E_\R\}^\frown \T$ is not a normal tree, and its full normalization $\U$ starts with $\T$. Notice also that for each $\a+1<\lh(\T)$, if $E^\T_\a$ used on the main branch of $\T$ at $\b<\lh(\T)$, then $\cp(E_\a^\T)<\pi_{0, \b}^\T(\d_\R)$. This is because $\R_E\models \card{o^{\R_E}(\k_\R)}=\d_\R$, as otherwise there would be another extender $F\in \vec{E}^{\R_E}\cap \vec{E}^\R$ such that $\cp(F)=\k_\R$ and $T_\R\in Ult(\R, F)$, contradicting the minimality of $E_\R$. Let $\a$ be such that $\T=\U\rest \a+1$. Our discussion shows that $\pi^\U_{0, \a}$ is defined. 

Set $\S'=\M_\a^\U$ and $h^*(n)=\pi^{\Lambda_{\S'}}_{\S', \infty}(\a_n)$. We now have that $(y, h')\in [T_{\S'}]$. This is simply because $T_{\S'}=U$. It then follows that $(y, h^*)\in [T]$, and hence $y\in p[T]$.
\end{proof}

\begin{lemma}\label{p2} For each $y\in p[T]$, there is a good $(\l, s)$ such that $y\in p[T_{\l, s}]$. 
\end{lemma}
\begin{proof} Fix $y\in p[T]$, and let $h\in \k^\omega$ be such that $(y, h)\in [T]$. We can find a complete $\Lambda$-iterate $\R$ of $\Q$ such that for each $n$, $h(n)\in \rge(\pi_{\R, \infty})$. Set $h'(n)=\pi^{-1}_{\R, \infty}(h(n))$. Let $(\l, s)=\pi_{\R_E, \infty}(\d_\R, s_\R)$. As $(y, h')\in [T_\R]$, we easily get that $(y, \pi_{\R_E, \infty}[h'])\in [T_{\l, s}]$. Hence, $y\in p[T_{\l, s}]$. 
\end{proof}
We thus have that $p[T]=\bigcup\{ p[T_{\l, s}]: (\l, s)\in \k\times \k^{<\omega}$ and $(\l, s)$ is good$\}$. Consider now the tree $U$ give by: $(u, h)\in U$ if and only if 
\begin{enumerate}
\item if $0\in \dom(h)$, then $h(0)=(\l_0, s_0)$ is good, and
\item if $\dom(u)=\dom(h)=_{def}m>1$, then \begin{center}  $(\la u(0),..., u(m-2)\ra, \la h(1), h(2),..., h(m-1) \ra)\in T_{\l_0, s_0}$.\end{center} 
\end{enumerate}
Then if $(x, h)\in [U]$, then $(x, h')\in [T]$, where $h'(n)=h(n+1)$. Also, if $(x, h)\in [T]$, then for some good $(\l, s)$ and $g\in \l^\omega$, $(x, g)\in [T_{\l, s}]$. Consequently, $(x, g')\in [U]$, where $g'(0)=(\l, s)$ and for $n\geq 1$, $g'(n)=h(n-1)$. We thus have that $p[T]=p[U]$, and as $U$ can be represented as a tree on $\omega\times \k$ and $\k<\d$, we get a contradiction to our assumption that $p[T]$ is not $\a$-Suslin for any $\a<\d$. 
\end{proof}

\section{On $X$-hod analysis}

\begin{definition}\label{completely total}
Suppose $\P$ is a premouse, an lbr premouse, or just some kind of hybrid premouse. Suppose $F\in \vec{E}^\P$. Then $F$ is \textbf{completely total} if $F$ is a total extender of $\P$ and $\nu(F)$ is a regular cardinal of $F$.
\end{definition}
\begin{definition}\label{omits f}
Suppose $\P$ is a premouse, an lbr premouse or just some kind of hybrid premouse. Suppose $\xi<\gg<{\sf{Ord}}\cap \P$, and $\T$ is an iteration tree on $\P$. Then $\T$ \textbf{omits} $(\xi, \gg)$ if 
whenever $\a<\lh(\T)$ is such that $\pi^{\T}_{0, \a}$ is defined, then
\begin{enumerate}
\item $\lh(E_\a^{\T})\not \in (\pi_{0, \a}^\T(\xi), \pi_{0, \a}^\T(\gg))$, and 
\item $\cp(E_\a^{\T_{\P, \Q}})\not \in (\pi_{0, \a}^\T(\xi), \pi_{0, \a}^\T(\gg))$.
\end{enumerate}
Suppose $F\in \vec{E}^\P$ is a completely total extender and $\T$ is an iteration tree on $\P$. Then $\T$ \textbf{omits} $F$ if $\T$ omits the interval $(\nu(F), \lh(F)+1)$.
\end{definition}

If $(\M, \Lambda)$ is a mouse pair and $F\in \vec{E}^\M$ is a completely total extender, we let $\mathcal{F}(\M, \Lambda, F)$\footnote{We can also define the set $\mathcal{F}(\M, \Lambda, (\xi, \gg))$ to be the set of those $\N$ such that $\T_{\M, \N}$ omits $(\xi, \gg)$ and prove most of the results of this section. However, we only need to develop this material for completely total extenders.} be the set of complete $\Lambda$-iterates $\N$ of $\M$  such that $\T_{\N, \M}$ omits $F$. 

\begin{definition}\label{e directed system}
Suppose $(P, \Psi, H, \a)$ is a coarse tuple and $x\in \dom(H)$. Set $H(x)=(\P, \Sigma)$, and suppose $F\in \vec{E}^\P$ is a completely total extender such that $\cp(F)$ is a cutpoint cardinal of $\P$. Let $F^+$, if it exists, be the extender on the extender sequence of $\P$ such that $\cp(F^+)=\cp(F)$, $\lh(F^+)>\lh(F)$, and $\lh(F)$ is a cutpoint of $Ult(\P, F^+)$.
Let $\P(F)=Ult(\P, F^+)$, if $F^+$ is defined, and otherwise $\P(F)=\P$.

Given $\Q, \R\in \mathcal{F}(\P(F), \Sigma, F)$, we set $\Q\leq_{\P, \Sigma, F} \R$ if $\R\in \mathcal{F}(\Q, \Sigma_\Q, \pi_{\P(F), \Q}(F))$.
\end{definition}

\begin{theorem}\label{e system is directed}
Continuing with the set up of \rdef{e directed system}, $\leq_{\P, \Sigma, F}$ is directed. 
\end{theorem}
\begin{proof}
We will use the following straightforward lemma.

\begin{lemma}\label{tracking a dissagreement} Suppose $(\P, \Sigma)$ is a mouse pair, and suppose $\T$ and $\U$ are two distinct iteration trees on $\P$ according to $\Sigma$ with last models $\Q$ and $\R$, respectively. Suppose further that $\gg$ is the least such that $\Q|\gg=\R|\gg$ and $\Q||\gg\not =\R||\gg$, and $\xi$ is such that $\T\rest \xi=\U\rest \xi$, but $\T\rest \xi+1\not =\U\rest \xi+1$. Assume $\lh(E_\xi^\T)<\lh(E_\xi^\U)$. Then  $\gg\in \dom(\vec{E}^\Q)\cup \dom(\vec{E}^\R)$, $\gg\not \in \vec{E}^\Q$, and $\vec{E}^\R(\gg)=E_\xi^\T$.
\end{lemma}

Fix $\Q, \R\in \mathcal{F}(\P, \Sigma, F)$, and let $\T=\T^\Sigma_{\P, \Q}$ and $\U=\T^\Sigma_{\P, \R}$. It follows from full normalization that the least-extender-disagreement comparison between $(\Q, \Sigma_\Q)$ and $(\R, \Sigma_\R)$ produces a common iterate $(\S, \Sigma_\S)$. We want to see that $\S\in \mathcal{F}(\P, \Sigma, F)$, which amounts to showing that $\Z=_{def}\T^\Sigma_{\P, \S}$ omits $F$.\footnote{In fact, we want to show that $\T_{\Q, \S}$ omits $\pi_{\P, \Q}(F)$ and $\T_{\R, \S}$ omits $\pi_{\P, \R}(F)$. These statements follow from the fact that $\S\in \mathcal{F}(\P, \Sigma, F)$. Indeed, if for example $\T_{\Q, \S}$ does not omit $\pi_{\P, \Q}(F)$, then the full normalization process described below would show that $\T_{\P, \S}$ does not omit $F$.} Set $\X=\T_{\Q, \S}$ and $\Y=\T_{\R, \S}$. Then $\Z$ is the full normalization of $\T^\frown \X$. 

We now assume that $\X$ and $\Y$ were built by allowing padding, so that $\lh(\X)=\lh(\Y)$, and our strategy is to analyze the full normalization process that produces $\Z$ out of $(\T, \X)$ and $(\T, \Y)$. We review some facts about the normalization process, and we do this for $(\T, \X)$. 

Recall that the full normalization process for $\T^\frown \X$ produces iteration trees $(\Z_\a: \a<\lh(\X))$ on $\P$, and $\Z$ is simply $\Z_{\lh(\X)-1}$ (e.g. see \cite{schlutzenberg2021normalization} or \cite{siskind2022normalization}). The sequence satisfies the following conditions.\\\\
(1.1) The last model of each $\Z_\a$ is $\M_\a^\X$. \\
(1.2) For $\a+1<\lh(\X)$, $\Z_{\a+1}$ is obtained by letting $\b$ be the $\X$-predecessor of $\a+1$, and \textit{minimally inflating} $\Z_\b$ by $E_\a^\X$. More precisely, letting $\gg_0$ be the least $\gg$ such that $E_\a^\X\in \vec{E}^{\M_\gg^{\Z_\a}}$ and $\gg_1$ be the least $\gg$ such that $\lh(E_\gg^{\Z_\b})>\cp(E_\a^\X)$, $\Z_{\a+1}\rest \gg_0+1=\Z_\a\rest \gg_0+1$, and for $\iota>0$ such that $\gg_0+\iota<\lh(\Z_{\a+1})$, $\M_{\gg_0+\iota}^{\Z_{\a+1}}=Ult(\M_{\gg_1+\iota-1}^{\Z_\b}, E_\a^\X)$. Also, for $\iota<\gg_0$, $E_\iota^{\Z_{\a+1}}=E_\iota^{\Z_\a}$, $E_{\gg_0}^{\Z_\a}=E_\a^\X$ and for $\iota>0$ such that $\gg_0+\iota<\lh(\Z_{\a+1})$, $E_{\gg_0+\iota}^{\Z_{\a+1}}$ is the last extender of $Ult(\M_\iota^{\Z_\b}||\lh(E_\iota^{\Z_\b}), E_\a^\X)$.\\
(1.3) Clause (1.2) above describes a natural embedding $\pi_{\b, \a+1}: \Z_\b\rightarrow \Z_{\a+1}$, a \textit{tree embedding}. If now $\a<\lh(\X)$ is a limit ordinal, then $\Z_\a$ is obtained as the direct limit of the system $(\Z_\b, \pi_{\b, \b'}: \b<\b', \b\in [0, \a)_\X, \b'\in [0, \a)_\X)$.\\\\
Now set $p=(\T, \X)$ and $q=(\U, \Y)$, and let $(\Z_\a^p, \Z_\a^q: \a<\lh(\X))$ be the two sequences produced by the respective normalization processes. To show that $\Z$ omits $F$, we inductively show that for $\a<\lh(\X)$, $\Z_\a^p$ and $\Z_\a^q$ omit $F$, and a close examination shows that the limit case is trivial. 

We now examine the successor stage of the induction. Suppose $\a+1<\lh(\X)$ is such that $\Z_\a^p$ and $\Z_\a^q$ omit $F$. We want to see that $\Z_{\a+1}^p$ and $\Z_{\a+1}^q$ also omit $F$. Let $\gg$ be such that $\Z_\a^p\rest \gg=\Z_\a^q\rest \gg$, $\gg+1\leq \max(\lh(\Z_\a^p), \lh(\Z_\a^q))$, and $E_\gg^{\Z_\a^p}\not = E_\gg^{\Z_\a^q}$. Assume without loss of generality that $\lh(E_\gg^{\Z_\a^q})<\lh(E_\gg^{\Z_\a^p})$. In this case, setting $G=E_\gg^{\Z_\a^q}$, we have that $\Z_{\a+1}^q=\Z_{\a}^q$ and $\Z_{\a+1}^p$ is the full normalization of $(\Z_{\a}^p)^{\frown}\{G\}$. 

Notice that since $G=E_\gg^{\Z_\a^q}$ and $\Z_\a^q$ omits $F$, $G$ cannot witness that $\Z_{\a+1}^p$ does not omit $F$. Also, because $\Z_{\a+1}^p\rest \gg+1=\Z_{\a}^p\rest \gg+1$, we have that $\Z_{\a+1}^p\rest \gg+2$ omits $F$. 

Fix some $\iota>0$, and let $\xi\leq \gg$ be the predecessor of $\gg+1$ in $\Z_{\a+1}^p$. Then $\M_{\gg+\iota}^{\Z_{\a+1}^p}=Ult(\M_{\xi+\iota-1}^{\Z_\a^p}, G)$, and $E_{\gg+\iota}^{\Z_{\a+1}^p}$ is the last extender of $Ult(\M_{\xi+\iota-1}||\lh(E_{\xi+\iota-1}^{\Z_\a^p}), G)$. Because $\Z_\a^p$ omits $F$, it is now straightforward to verify that $\Z_{\a+1}^p\rest \gg+\iota+1$ omits $F$.
\end{proof}

The preceding proof can be modified to show the following corollary.

\begin{corollary}\label{comparing above} Suppose $(\P, \Sigma)$ is a mouse pair and $\eta<{\sf{Ord}}\cap \P$. Suppose $\Q$ and $\R$ are two $\Sigma$-iterates of $\P$ such that both $\T_{\P, \Q}$ and $\T_{\P, \R}$ are strictly above $\eta$. Then the least-extender-comparison of $(\Q, \Sigma_\Q)$ and $(\R, \Sigma_\R)$ produces iteration trees that are strictly above $\eta$.
\end{corollary}
Suppose $(P, \Psi, H, \a)$ is a coarse tuple and $x\in \dom(H)$. Set $(\P, \Sigma)=(\P_x, \Sigma_x)$. Let $(\d^{i}_\P: i\leq \omega)$ be the sequence of Woodin cardinals of $\P$ and their limit. If $\Q$ is a complete $\Sigma$-iterate of $\P$, then we write $\d^i_\Q=\pi_{\P, \Q}(\d^i)$. 

Suppose $\mu<\d^0_\P$ is a measurable cutpoint of $\P$ such that $\mu$ is below the least $<\d_0$-strong cardinal of $\P$ and $F\in \vec{E}^\P$ is a completely total extender with $\cp(F)=\mu$. Set $\P_F=Ult(\P, F)$, and let $\k_{\P, \Sigma, F}=\pi_{\P_F, \infty}(\mu)$ and  $\tau_{\P, \Sigma, F}=o^{\M_\infty(\P, \Sigma)}(\k_{\P, \Sigma, F})$. Let $\T_{\P, \Sigma, F}=_{def}\T$ be the least initial segment of $\T_{\P, \M_\infty(\P, \Sigma)}$ such that if $\R$ is the last model of $\T$, then $\R|\tau_{\P, \Sigma, F}=\M_\infty(\P, \Sigma)|\tau_{\P, \Sigma, F}$.
We then set $E_{\P, \Sigma, F}=_{def}E=\pi^\T(F)$,\footnote{$\pi^\T$ is defined because $F$ is completely total and $\T$ is based on $\P|\nu(F)$.} $X_{\P, \Sigma, F}=\R||\lh(E)=_{def} X$, and $\mH(\P, \Sigma, F)=({\sf{HOD}}_{\Psi,X}|\Theta)^W$.

\begin{theorem}\label{hod e}  In the above set up, setting $W=L_\a^\Psi(\bR)$, 
\begin{center}
$\M_\infty(\P, \Sigma, F)|\Theta^W=\mH(\P, \Sigma, F)|\Theta^W$.\footnote{As is usual, here we mean that the universes of the structures are the same.} 
\end{center}
\end{theorem}
\begin{proof} The argument is somewhat standard, so we will only give an outline. The key fact to keep in mind is that if $\T$ is an iteration tree on $\P(F)$ which omits $F$, then $\T$ can be split into two components $\T_l^\frown \T_r$ such that $\T_l$ is based on $\P|\nu(F)$, and if $\T_l\not =\T$, then $\pi^{\T_l}$ is defined and $\T_r$ is strictly above $\lh(\pi^{\T_l}(F))$. Below and elsewhere, $F_\Q=\pi^{\T_{\P(F), \Q}}(F)$.
\begin{claim}\label{internal direct limits} Suppose $j:\P(F)\rightarrow \Q$ is an $\bR$-genericity iteration according to $\Sigma$ in which all iteration trees used omit $F$. Then 
\begin{enumerate}
\item $X_{\P, \Sigma, F}\in \Q$, 
\item $X_{\P, \Sigma, F}$ is a $\Sigma_{\Q||\lh(F_\Q)}$-iterate of $\Q||\lh(F_\Q)$, and
\item $\T_{\Q||\lh(F_\Q), X_{\P, \Sigma, F}}\in \Q$. 
\end{enumerate}
\end{claim}
We can then develop the concept of suitable premouse, $A$-iterable suitable premouse, and other concepts used in the HOD analysis for iterations that omit $F$. For example, we define $\S$ to be suitable if, in addition to the usual properties of suitability (see \cite{HODCoreModel}), for some $G\in \vec{\S}$, $X_{\P, \Sigma, F}$ is a complete iterate of $\S||\lh(G)$. $A$-iterability is defined for those $A\subseteq \bR$ which are ordinal definable from $X_{\P, \Sigma, F}$, and given a suitable $\S$ we define the concept of $A$-iterability only for those iterations of $\S$ that omit $G$, where $G$ is as above. \rcl{internal direct limits} can now be used to show that for every $A\subseteq \bR$ that is ordinal definable from $X_{\P, \Sigma, F}$, there is a strongly $A$-iterable pair. The rest is just like in the ordinary HOD analysis, and we leave it to the reader. 
\end{proof}
For $(\P, \Sigma, F)$ as above, let $\Sigma^F$ be the fragment of $\Sigma_{\P(F)}$ that acts on stacks that omit $F$, and let $\pi_{\P(F), \infty}^{\Sigma^F}:\P(F)\rightarrow \M_\infty(\P, \Sigma, F)$ be the direct limit embedding.

\section{Regular Suslin cardinals are $\omega$-club $\Theta$-Berkeley}

\begin{theorem}\label{main theorem} Assume ${\sf{AD^+}}$. Then every regular Suslin cardinal is $\omega$-club $\Theta$-Berkeley.
\end{theorem}
\begin{proof}
Fix a regular Suslin cardinal $\d$, and towards a contradiction assume that $\d$ is not an $\omega$-club $\Theta$-Berkeley cardinal. Fix a transitive $N'$ such that\\\\
(1.1) $\card{N'}<\Theta$, \\
(1.2) $\delta\subseteq N'$, and \\
(1.3) the set of $\a<\d$ such that there is no elementary embedding $j: N'\rightarrow N'$ with the property that $\cp(j)=\a$ is $\omega$-stationary in $\d$.\footnote{I.e., intersects every $\omega$-club set.}\\\\
Let $\phi(u, v)$ be the formula expressing (1.1)-(1.3). Thus $\phi[\d, N']$ holds.
Since $L(\powerset(\bR))\models \phi[\d, N']$, we can assume that $V=L(\powerset(\bR))$. Without loss of generality, assume that $\d$ is the least regular Suslin cardinal $\k$ such that $\phi[\k, M]$ holds for some $M$. It then immediately follows that $\d$ cannot be the largest Suslin cardinal, as if $\d$ were the largest Suslin cardinal, then for some $\a<\d$ and some $\b<\d$, letting $\Delta=\{A\subseteq \bR: w(A)<\b\},$\footnote{Where $w(A)$ is the Wadge rank of $A$.} $L_\a(\Delta)\models {\sf{ZF-Replacement}}+\exists \k\exists M\phi[\k, M]$.\footnote{We actually need the Coding Lemma here. Let $(\k, M)$ witness $\phi$ in $L_\a(\Delta)$. Let $S\subseteq \k$ consist of ordinals $\gg<\k$ such that there is no elementary embedding $j: M\rightarrow M$ with $\cp(j)=\gg$ and $j\in L_\a(\Delta)$. Then $S$ is stationary in $L_\a(\Delta)$ and hence, since $\powerset(\k)\subseteq L_\a(\Delta)$, $S$ is stationary in $V$. The fact that there is no $j: M\rightarrow M$ in $L_\a(\Delta)$ with $\cp(j)\in S$ implies that there is no $j$ with $\cp(j)\in S$. Hence, $V\models \phi[\k, M]$.} A similar reflection argument shows that we can assume without losing generality that $\card{N'}$ is less than some Suslin cardinal $\d'$ such that $\d<\d'$ and $\d'$ is not the largest Suslin cardinal. Fix now such a $\d'$ so that $\card{N'}<\d'$. 

Let $(R, \Psi, H, \a)$ be as in \rthm{woodins thm} absorbing $\d'$. Notice that our discussion above implies that $W\models \phi[\d, N']$, for $W=L_\a^\Psi(\bR)$. We can then find $x_0$ such that $N'$ is ${\sf{OD}}^W_{\Psi, x_0}$, and then by minimizing, we can let $N$ be the ${\sf{OD}}^W_{\Psi, x_0}$-least $M$ such that $W\models \phi[\d, M]$. We now fix $x_1\in \dom(H)$ such that\\\\
(2.1) $x_1$ is Turing above $x_0$ and $N\in \M_\infty(\P_{x_1}, \Sigma_{x_1})$, where $(\P_{x_1}, \Sigma_{x_1})=H(x_1)$, and\\
(2.2) $\d$ is a limit of cutpoints of $\M_\infty(\P, \Sigma)$, where $(\P, \Sigma)=(\P_{x_1}, \Sigma_{x_1})$.\\\\
Without loss of generality, we assume that $\d\in \rge(\pi^\Sigma_{\P, \infty})$, and for $\Q$ a complete $\Sigma$-iterate of $\P$, we let $\d_\Q=\pi_{\Q, \infty}^{-1}(\d)$. 

Fix now any completely total extender $F\in \vec{E}^\P$ such that $\cp(F)=\d_\P$, and set $X_\P=X_{\P, \Sigma, F}$. Let $\Q=Ult(\P(F), F)$, $F_\Q=\pi_F(F)$ and $X_\Q=X_{\Q, \Sigma_\Q, F_\Q}$. We set $\Lambda=\Sigma^F$ and $\Phi=\Sigma^{F_\Q}_\Q$.
As in the proof of \rthm{main theorem lr}, we define $j: \M_\infty(\P, \Sigma, F)\rightarrow \M_\infty(\Q, \Sigma_\Q, F_\Q)$. It follows from \rthm{hod e} and the proof of \rcl{jnn} that $N\in \M_\infty(\P, \Sigma, F)\cap \M_\infty(\Q, \Sigma_\Q, F_\Q)$ and $j(N)=N$, and therefore defining $j$ is all that we will do. 

Fix $u\in \M_\infty(\P, \Sigma, F)$, and let $\S$ be a complete $\Lambda$-iterate of $\P(F)$ such that $u=\pi_{\S, \infty}^{\Lambda}(u_\S)$ for some $u_\S\in \S$. 
Let $F_\S=\pi_{\P(F), \S}(F)$\footnote{We will use this notation for all iterates of $\P$.} and $\S_F=Ult(\S, F_\S)$. We then let\footnote{In this section,  if we write $\pi_G$ then we tacitly assume that it is the ultrapower embedding obtained by applying $G$ to the model it is chosen from. If we apply $G$ to some other model $\N$, then we will write $\pi_G^\N$.} 
\begin{center}
$j(u)=\pi_{\S_F, \infty}^{\Phi}(\pi_{F_\S}(u_\S))$.
\end{center}
The definition of $j(u)$ makes sense, as full normalization implies that $\S_F$ is a complete $\Phi$-iterate of $\Q$. To prove this and other claims in this section, we set $\P=\P(F)$ to simplify the notation. 
\begin{claim}\label{j makes sense} The definition of $j$ is meaningful; more precisely, $\S_F$ is a compete $\Phi$-iterate of $\Q$.
\end{claim}
\begin{proof}
Notice that $\T_{\P, \S}$ can be split into $(\T^l_{\P, \S})^\frown  \T^r_{\P, \S}$ where $\T^l_{\P, \S}$ is the longest portion of $\T_{\P, \S}$ that is based on $\P|\nu(F)$ and $\T^r_{\P, \S}$ is the rest of $\T_{\P, \S}$. If $\T^r_{\P, \S}$ is defined, then it is above $\lh(\pi^{\T^l_{\P, \S}}(F))$. It then follows that the full normalization of $(\T_{\P, \S})^\frown \{F_\S\}$\footnote{Here $F_\S$ is applied to $\S$.} is $(\T^l_{\P, \S})^\frown\{ F_\S\}^\frown \U$,\footnote{Here $F_\S$ is applied in a way that keeps the iteration normal.} where $\U$ is above $\lh(F_\S)$. Notice next that the full normalization of $\{F\}^\frown \T^l_{\P, \S}$ is $(\T^l_{\P, \S})^\frown\{ F_\S\}$. Thus, $\S_F$ is a $\Phi$-iterate of $\Q$ and $\T_{\Q, \S_F}=(\T^l_{\P, \S})^\frown \U$.

\end{proof}
\begin{claim}
$j(u)$ is independent of the choice of $\S$.
\end{claim}
\begin{proof}
To see this, pick another normal $\Lambda$-iterate $\S'$ of $\P$ such that $\pi_{\S', \infty}^{\Lambda_{\S'}}(u_{\S'})=u$. It then follows from \rlem{e system is directed} that we can compare $(\S, \Lambda_\S)$ and $(\S', \Lambda_{\S'})$ via the least-extender-disagreement process and get some common iterate $(\S'', \Lambda_{\S''})$. It then follows that $\pi_{\S, \S''}(u_S)=\pi_{\S', \S''}(u_{\S'})=u_{\S''}$. 

Consider now $\S_F, \S'_F$, and $\S''_F$. We want to see that 
\begin{center}
(*)\ \ \ $\pi_{\S_F, \infty}^{\Phi}(\pi_{F_\S}(u_\S))=\pi_{\S'_F, \infty}^{\Phi}(\pi_{F_{\S'}}(u_{\S'}))$.
\end{center}
To show (*), we observe that \\\\
(3.1) $\S''_F$ is a complete $\Phi_{\S_F}$-iterate of $\S_F$,\\
(3.2) $\S''_F$ is a complete $\Phi_{\S'_F}$-iterate of $\S'_F$,\\
(3.3) $\pi_{\S_F, \S''_F}(\pi_{F_\S}(u_\S))=\pi_{\S'_F, \S''_F}(\pi_{F_{\S'}}(u_{\S'}))$\\\\
Notice that (3.3) implies (*). (3.3) is an immediate consequence of (3.1) and (3.2), and (3.2) has the same proof as (3.1), and (3.1) follows from the proof of \rcl{j makes sense}.  
\end{proof}

\begin{claim} $\cp(j)=\cp(E_{\P, \Sigma, F})$
\end{claim}
\begin{proof} Suppose $\a<\cp(E_{\P, \Sigma, F})$. Let $\S$ be a $\Phi$-iterate of $\P$ such that $\T_{\P, \S}$ is based on $\P|\nu(F)$ and $\a\in \rge(\pi^\Sigma_{\P, \S})$. We then have that if $\a_\S=(\pi^{\Sigma}_{\P, \S})^{-1}(\a)$, then $\pi^\Lambda_{\S, \infty}(\a_\S)=\a$. Notice next that $j(x)=\pi^\Phi_{\S_F, \infty}(\pi_{F_\S}(\a_\S))$, and since $\pi_{F_\S}(\a_\S)=\a_\S$,\footnote{This is where we use that $\a<\cp(E_{\P, \Sigma, F})$.} setting $\W=\S_F|\nu(\pi_{\P, \S}(F))$, $\pi^{\Phi}_{\S_F, \infty}(\a_\S)=\pi^{\Sigma_{\W}}_{\W, \infty}(\a_\S)$. Since $\W=\S|\nu(\pi_{\P, \S}(F))$, we have that $\pi^{\Phi}_{\S_F, \infty}(\a_\S)=\a$, implying that $j(\a)=\a$.
\end{proof}

To finish the proof of \rthm{main theorem}, we need to produce an $\omega$-club $C\subseteq \omega_1$ such that for each $\a\in C$, there is an embedding $k: N\rightarrow N$ with $\cp(k)=\a$. Above, we have produced an elementary embedding $j_{\P, \Sigma, F}:N\rightarrow N$ such that $\cp(j_{\P, \Sigma, F})=\cp(E_{\P, \Sigma, F})$. We then apply this fact to the Mitchell order 0 extender $F$ such that $\cp(F)=\d_\P$. Let $C$ consist of ordinals $\k$ such that for some complete $\Sigma$-iterate $\Q$ of $\P$, $\k=\cp(E_{\Q, \Sigma_\Q, F_\Q})$. Then $C$ is an $\omega$-club\footnote{See Steel's proof of measurability of regular cardinals, \cite[Chapter 8]{OIMT}. This is easier to establish with full normalization. Each $\k\in C$ is on the main branch of $\T_{\P, \M_\infty(\P, \Sigma)}$ and is non-measurable in $\M_\infty(\P, \Sigma)$. It follows that the exit extender used is the image of $F$.} and is such that for each $\k\in C$, there is $j: N\rightarrow N$ such that $\cp(j)=\k$. This finishes the proof of \rthm{main theorem}.
\end{proof}

\begin{remark}\label{omega2}\normalfont The proof of \rthm{main theorem} demonstrates that $\omega_2$ is $\Theta$-Berkeley. Indeed, fix some $\eta<\omega_2$, and pick $x\in \dom(H)$ such that if $\tau_x$ is the second  measurable cardinal of $\P_x$, then $\pi_{\P_x, \infty}(\tau_x)>\eta$. Let $F\in \vec{E}^{\P_x}$ be the Mitchell order 0 extender with $\cp(F)=\tau_x$. We now repeat the proof of \rthm{main theorem} and get that  $\pi_{\P_x, \infty}(\tau_x)$ is a limit of ordinals $\a$ such that there is a $j: N\rightarrow N$ with $\cp(j)=\a$. Since ordinals of the form $\pi_{\P_x, \infty}(\tau_x)$ are cofinal in $\omega_2$, we get that $\omega_2$ is a $\Theta$-Berkeley cardinal. 

The same argument can be used to show that for any $n$, $\utilde{\d}^1_{2n}$ is $\Theta$-Berkeley. 
This is because for each $n$ and for each $x\in \dom(H)$, $\M_\infty(\P_x, \Sigma_x)$ has a cutpoint cardinal that belongs to the interval $(\utilde{\d}^1_{2n+1}, \utilde{\d}^1_{2n+2})$,\footnote{E.g., see \cite{Sa13}.} and $\utilde{\d}^1_{2n+2}$ is a limit of measurable cardinals of $\M_\infty(\P_x, \Sigma_x)$. If now $\k_x>\utilde{\d}^1_{2n+1}$ is the least cutpoint measurable of $\M_\infty(\P_x, \Sigma_x)$, then the proof of \rthm{main theorem} shows that for unboundedly many $\a<\k_x$, there is $j: N\rightarrow N$ such that $\cp(j)=\a$. Since $\utilde{\d}^1_{2n+2}$ is a limit of ordinals of the form $\k_x$, we have that $\utilde{\d}^1_{2n+2}$ is $\Theta$-Berkeley. 
\end{remark}

\section{Towards $\sf{HOD}$-Berkeley cardinals}

\rrem{omega2} leaves open how ubiquitous $\Theta$-Berkeley cardinals are.
\begin{question} Assume ${\sf{AD^+}}$. Is there an uncountable cardinal $\k<\Theta$ that is \emph{not} $\Theta$-Berkeley? Is every regular cardinal club $\Theta$-Berkeley?
\end{question}

A regular cardinal $\k$ is a \textit{club} $({\sf{OD}}, \l)$-Berkeley if for every $x\in H_\k$ and every transitive structure $M$ of size $<\l$ such that $M$ is ordinal definable from $x$, there is a club $C\subseteq \k$ such that for each $\a\in C$ there is an elementary embedding $j: M\rightarrow M$ with $\cp(j)=\a$.
The following is an easy corollary to \rthm{main theorem}.
It follows from the fact that $\mathbb{P}_{{\sf{max}}}*Add(1, \omega_3)$ is a countably closed homogeneous poset. 

\begin{corollary}\label{main corollary} 
Assume ${\sf{AD}}_{\bR}+V=L(\powerset(\bR))+``\Theta$ is a regular cardinal". Let $G\subseteq \mathbb{P}_{{\sf{max}}}*Add(1, \omega_3)$ be $V$-generic. Then in $V[G]$, 
$\omega_1$ is club $({\sf{OD}}, \omega_3)$-Berkeley.
\end{corollary}

Obtaining a model of $\mathsf{ZFC}$ + ``there is a $\mathsf{HOD}$-Berkeley cardinal'' by forcing seems like a hard problem. In this direction, Gabriel Goldberg has shown the following proposition. We include his argument with permission.\footnote{Goldberg observed that his argument works with the ostensibly weaker hypothesis ``for all sufficiently large ordinals $\gamma$, there is an elementary embedding $j:V_\gamma^{\sf{HOD}}\to V_\gamma^{\sf{HOD}}$.''}

\begin{proposition}[Goldberg]\label{gabe} Suppose there is a ${\sf{HOD}}$-Berkeley cardinal. Then $A^{\#}$ exists for all sets $A$.
\end{proposition}
\begin{proof}
Since every set of ordinals is set generic over HOD, it is enough to show that every set of ordinals that belongs to HOD has a sharp.
Let $A$ be a set of ordinals in $\sf{HOD}$ with $\sup A =\lambda$, and
let $\gamma$ be a $\Sigma_2$-correct ordinal $>\lambda$ of uncountable cofinality.
Let $j:V_{\gamma}^{\sf{HOD}}\to V_{\gamma}^{\sf{HOD}}$ be an elementary embedding.
(Notice $j$ is not definable over $V$; if it were, it would belong to a ${<}\gamma$-generic extension of $\sf{HOD}$, contrary to Woodin's proof \cite[p.~320]{kanamori2008higher} of the Kunen Inconsistency.)
Letting $E$ be the extender of length $\lambda+1$ derived from $j$, we have that $j$ factors into embeddings $j_E: V_{\gamma}^{\sf{HOD}}\to M$ and $k:M\to V_{\gamma}^{\sf{HOD}}$.
Since $j$ is not definable and hence is not the extender ultrapower, $k$ must be nontrivial with $\cp(k)>\lambda$.
Then $k:L_{\gamma}[A]\to L_{\gamma}[A]$ is a nontrivial elementary embedding with $\cp(k)>\lambda$.
\end{proof}

With a stronger hypothesis, we can get $\M_1^\#$.

\begin{proposition}\label{core model} Suppose there is a ${\sf{HOD}}$-Berkeley cardinal and a measurable cardinal above it.
Then $\M_1^\#$ exists and is ${\sf{Ord}}$-iterable. 
\end{proposition}
\begin{proof} Let $\iota$ be a measurable cardinal above the least ${\sf{HOD}}$-Berkeley cardinal. We first show that the core model $K=_{def}K^{V_\iota}$ does not exist. Towards a contradiction, assume that it does.  Since $K\in\sf{HOD}$, we have a non-trivial embedding $j: K\rightarrow K$. But then \cite[Theorem 8.8]{CMIP} gives a contradiction. It now follows from the same aforementioned theorem that in fact $(K^c)^{V_\iota}\models``$there is a Woodin cardinal," and since V is closed under sharps by \rprop{gabe}, we get that $\M_1^\#\insegeq (K^c)^{V_\iota}$. Because $V$ is closed under sharps, it follows that $\M_1^\#$ is ${\sf{Ord}}$-iterable (see \cite{HODCoreModel}). 
\end{proof}

Combining the arguments for \rprop{gabe} and \rprop{core model}, we get some definable determinacy.

\begin{theorem}\label{pd} Suppose there is a ${\sf{HOD}}$-Berkeley cardinal and a class of measurable cardinals. Then Projective Determinacy holds.
\end{theorem}
\begin{proof} The proof is via the core model induction as in \cite{St05}. We show that $\M_2^\#$ exists and leave the rest to the reader. To show that $\M_2^\#$ exists, it is enough to show that $V$ is closed under the $\M_1^\#$-operator and $K^c\models ``$ There is a Woodin cardinal". The second statement is very much like the proof of \rprop{core model}, and so we only show that $V$ is closed under the $\M_1^\#$-operator. 

As in the proof of \rprop{gabe}, it is enough to show that for every set of ordinals $A\in {\sf{HOD}}$, $\M_1^\#(A)$ exists, and to show this, it is enough to show that for every $A\in {\sf{HOD}}$, $K(A)$ does not exist. 

Fix now $A\in {\sf{HOD}}$, $A\subseteq \l$, and let $\iota>\l$ be a measurable cardinal above the least ${\sf{HOD}}$-Berkeley cardinal. Assume that  $K=K(A)^{V_\iota}$ exists. Notice that $K\in {\sf{HOD}}$. Let $\gg>\iota$ be a $\Sigma_2$-correct cardinal, and let $M\in {\sf{HOD}}$ be such that, letting $a=\{V_\gg^{\sf{HOD}}, K\}$, $a\in M$ and $a$ is definable in $M$ (see \cite[Lemma 3.1]{bagaria2019large}). Let $j': M\rightarrow M$ be non-trivial and elementary. 

Let $j=j'\rest V_\gg^{\sf{HOD}}$. Notice that $j(K)=K$. Let $\k=\cp(j)$, and let $F$ be the $(\k, \l)$-extender derived from $j$. As in the proof of \rprop{gabe}, if $k': Ult(V_\gg^{\sf{HOD}}, F)\rightarrow V_\gg^{\sf{HOD}}$ is the canonical factor map, then $\cp(k')>\l$. Let then $\M=\pi_F(K)$, and set $\pi_F=i$ and $k'\rest \M=k$. We thus have that $i:K\rightarrow \M$ and $k:\M\rightarrow K$. Moreover, $\cp(k)>\l$. 

Because $i:K\rightarrow \M$, it follows that $\M$ is universal among $A$-mice of ordinal height $\iota$, and therefore there is $\sigma: K\rightarrow \M$ such that $\cp(\sigma)>\l$ (see \cite{CMIP}). It follows that $k\circ \sigma :K\rightarrow K$, and so we can get a contradiction as in \rprop{core model}.
\end{proof}

These theorems show that obtaining ${\sf{HOD}}$-Berkeley cardinals requires significant large cardinals. We believe that the proof of \rthm{pd} can be extended to show that $L(\bR)\models {\sf{AD}}$ and the hypothesis that there is a class of measurable cardinals is unnecessary (see \cite{JS13}). But establishing these beliefs is beyond the scope of this paper, and we conclude this discussion with the following conjecture.

\begin{conjecture} Suppose there is a ${\sf{HOD}}$-Berkeley cardinal. Then the minimal model of ${\sf{AD}}_{\bR}+``\Theta$ is a regular cardinal" exists. 
\end{conjecture}

\section{\rthm{main theorem lr} is optimal}

In this section we use the main idea of \rprop{gabe}, ideas from \cite{JSSS}, and the HOD analysis of $L(\bR)$ (see \cite{HODCoreModel}) to show that \rthm{main theorem lr} cannot be improved, assuming that $V=L(\bR)$. 

\subsection{\textbf{Hod-like pairs.}}
Our strategy for proving \rthm{lr optimal} is the following. Assume $V=L(\powerset(\bR))+{\sf{AD}}$, and suppose there is an embedding $j:{\sf{HOD}}|\Theta\rightarrow {\sf{HOD}}|\Theta$. We want to show that $j$ can be extended to $j^+:{\sf{HOD}}\rightarrow {\sf{HOD}}$. Via the reasoning of \rprop{gabe}, this leads to a contradiction.

To implement our strategy, we need to use more of the HOD analysis than the previous sections required. The HOD analysis that we need is developed in \cite[Chapter 6]{HODCoreModel}, in particular \cite[Theorem 6.1]{HODCoreModel}. Recall from \cite[Theorem 6.1]{HODCoreModel} that assuming $V=L(\bR)$, ${\sf{HOD}}=L[\M_\infty^+, \Lambda]$. While \cite[Theorem 6.1]{HODCoreModel} is proved assuming $\M_\omega^\#$ exists, the proof can also be done by first reflecting and then picking a coarse tuple as we have done in the arguments presented in the previous sections (see \rdef{capturing a cardinal}). The proof simply needs a pair $(\P, \Sigma)$ whose derived model is $L(\bR)$, or $L_\alpha(\bR)$, as we will do below.

The exact meaning of $\M_\infty^+$ and $\Lambda$ are very important for us, and we will set up some notation to discuss these object.

\begin{notation}\label{woodins} Suppose $\P$ is a premouse. Then $(\d^\a_\P: \a\leq \iota)$ denotes the increasing enumeration of the Woodin cardinals of $\P$ and their limits, and $\eta^\a_\P$ denotes the $\P$-successor of $\d^\a_\P$, if it exists. 
\end{notation}

\begin{definition}\label{omega woodins} Suppose $\P$ is a premouse with exactly $\omega$ many Woodin cardinals. Let $\mH_\P$ be the premouse representation of $({\sf{HOD}}|\Theta)^{D(\P, \d^\omega_\P)}$.\footnote{$D(\P, \d^\omega_\P)$ is the derived model of $\P$.} 

We say $\P$ is \textbf{hod-like} if 
\begin{enumerate}
\item $\P\models {\sf{ZFC-Replacement}}$, and
\item there is a tree $\T\in \P$ of limit length such that $\cop(\T)=\mH_\P$ and $\T$ is based on $\P|\d^0_\P$.\footnote{$\cop(\T)$ is the common part of $\T$ which usually is denoted by $\M(\T)$. Since $\M$ is overused in inner model theory, we will use $\cop(\T)$.}
\end{enumerate}
 If $\P$ is hod-like, then we let $\T_\P$ be the normal tree $\T$ such that $\cop(\T)=\mH_\P$.

Suppose $\P$ is a hod-like and $b$ is a branch of $\T_\P$. We say $b$ is \textbf{friendly} to $\P$ if $\pi^\T_b(\d^0_\P)=\Theta^{D(\P, \d^\omega_\P)}$. 

Suppose $\P$ is hod-like and $b$ is friendly to $\P$. We then let $\xi_\P=\lh(\T_\P)$, $b(\P)=\{\eta^\omega_\P+i: i\in b\}$ and
\begin{center}
$\V'(\P, b)=(\P|\eta^\omega_\P+\xi_\P, b(\P))$.
\end{center}
We then say $\V'(\P, b)$ is the \textbf{pre-Varsovian model}\footnote{See \cite{SaSch18}.} induced by $(\P, b)$. We say $(\P, b)$ is hod-like if the core of $\V'(\P, b)$ is defined\footnote{See \cite{OIMT}.} and both its $\Sigma_1$-projectum and projectum are $\d^0_\P$. If $(\P, b)$ is hod-like, then we let $\V(\P, b)$ be the core of $\V'(\P, b)$. 
\end{definition}
We treat $\V(\P, b)$ as a hybrid premouse, see \cite{SchlT}. Next we introduce hod-like pairs.

\begin{definition}\label{res hod like pairs} We say $(\P, \Sigma)$ \textbf{resembles a hod-like pair} if it is a pure mouse pair such that the following conditions hold:
\begin{enumerate}
\item $\Sigma$ is an $\omega_1+1$-iteration strategy.
\item $\P$ is hod-like.
\item If $b=\Sigma(\T_\P)$ then $(\P, b)$ is hod-like.
\end{enumerate}
\end{definition}
To turn pairs that resemble hod-like pairs into true hod-like pairs, we need to impose some conditions which fall naturally out of the HOD analysis. 

\begin{definition}\label{hod like pairs} A pair $(\P, \Sigma)$ is \textbf{a hod-like pair} if it resembles a hod-like pair and the following conditions hold:
\begin{enumerate}
\item \textbf{Self-capturing:} For every cutpoint successor cardinal $\nu$ of $\P$ and ordinal $\gg<\nu$, if $\P$ has no Woodin cardinals in the interval $(\gg, \nu)$, then the fragment of $\Sigma$ that acts on iterations that are based on $\P|\nu$ and are above $\gg$ is in the derived model of $(\P, \Sigma)$. 
\item \textbf{Self-similar:} For every $i\in [1, \omega)$, if $\W$ is the output of the fully backgrounded construction of $\P|\d^i_\P$ in which all extenders used have critical points $>\d^{i-1}_\P$, then $\W$ is a $\Sigma_{\P|\d^0_\P}$-iterate of $\P|\d^0$.
\end{enumerate}
Suppose now that $\P$ is hod-like with exactly $\omega$ many Woodin cardinals. Then $\P$ is \textbf{self-similar} if for every $i\in \omega$ and $\gg\in [\d^i_\P, \d^{i+1}_\P)$, letting $\W^i$ be the output of the fully backgrounded construction of $\P|\d^{i+1}_\P$, there is a normal iteration tree $\T^i\in \P$ on $\P$ such that $\T^i$ is based on $\P|\d^0_\P$ and $\cop(\T^i)=\W^i$.
\end{definition}
Lastly, we introduce the abstract Varsovian models and self-determining Varsovian models.
\begin{definition}\label{varsovian models} We say $\V$ is a \textbf{Varsovian model} if for some hod-like $(\P, b)$ with $\P$ self-similar, $\V=\V'(\P, b)$. If $\V$ is a Varsovian model witnessed by $(\P, b)$, then we let $\X^\V=\P$, $\U^\V=\T_\P$, $\mH^\V=\mH_\P$, $b^\V=b$, and for $i\in \omega$, $(\W_i^\V, \T_i^\V)=(\W_i, \T_i)$ where $(\W_i, \T_i)$ is as in \rdef{hod like pairs}.
\end{definition}

\begin{definition}\label{self determining v-models} Suppose $\V=\V'(\P, b)$ is a Varsovian model. Then $\V$ is \textbf{self-determining} if for each $i\in \omega$, letting $(\W_i, \T_i)=(\W_i^\V, \T_i^\V)$, there is $\U_i\in \P$ such that $\cop(\U_i)=\mH^\V$ and there is a unique pair of branches $(c_i, d_i)$ such that $\pi^{\U^\V}_{b^\V}=\pi^{\U_i}_{d_i}\circ \pi^{\T_i}_{c_i}$.

In the above situation, we let $(\U_i, c_i, d_i)=(\U_i^\V, c_i^\V, d_i^\V)$.
\end{definition}

\begin{definition}\label{varsovian pair} We say that $(\V, \Lambda)$ is a \textbf{Varsovian pair} if $\V=\V'(\P, b)$ is a self-determining Varsovian model and $\Lambda$ is an iteration strategy for $\V$ such that, whenever $\V'$ is a complete $\Lambda$-iterate of $\V$, $\V'$ is self-determining, all the iteration trees $\U^{\V'}, \U_i^{\V'}, \T_i^{\V'}$ and the associated branches $b^{\V'}, c_i^{\V'}, d^{\V'}_i$ are according to $\Lambda_{\V'|\d^0_{\V'}}$.
\end{definition}

\begin{definition}\label{iteration of v-models} Suppose $\V=\V'(\P, b)$ is a Varsovian model and $\Gamma$ is an iteration strategy for $\P|\d^0_\P$. Then $(\V, \Lambda)$ is a $\Gamma$-\textbf{Varsovian pair} if $\V$ has a $\card{\V}^++1$-iteration strategy $\Lambda$ such that $\Lambda_{\V|\d^0_\V}=\Gamma$.

We say $\V$ is $\Gamma$-Varsovian model if there is a unique $\Lambda$ such that $(\V, \Lambda)$ is a $\Gamma$-Varsovian pair.
\end{definition}

The following useful lemma is easy to verify and we leave it to the reader.

\begin{lemma}\label{uniquness of strategies} Suppose $(\V, \Lambda)$ is a Varsovian pair and $i\in \omega$. Let $\T$ be an iteration tree on $\V$ according to $\Lambda$ such that $\T$ has a limit length, $\T$ is based on $(\d^i_\V, \d^{i+1}_\V)$,\footnote{I.e. is above $\d^i_\V$ and is based on $\V|\d^{i+1}_\V$.} and letting $a=\Lambda(\T)$, $a$ is non-dropping. Then for every $\a<\d^0_\V$ and $k\in \omega$, \begin{center}
$\pi^{\T}_a(\pi^{\T^\V_i}_{c^\V_i}(\a))=\pi^{\T^{\V'}_i}_{c^{\V'}_i}(\a)$.
\end{center}
Hence, if $\pi^\T_a(\d^{i+1}_\V)=\d(\T)$, then $a$ is the unique branch $e$ of $\T$ such that $\pi^{\T^{\V'}_i}_{c^{\V'}_i}[\d^0_\V]\subseteq \rge(\pi^\T_e)$.
\end{lemma}
The last clause of \rlem{uniquness of strategies} is important because it shows that $\Lambda_{\V|\d^0_\V}$ determines $\Lambda$. Indeed, $\pi^{\T^{\V'}_i}_{c^{\V'}_i}$ depends only on $\Lambda_{\V|\d^0_\V}$ and $\cop(\T)$. We thus have that if $(\V, \Lambda)$ is a Varsovian pair, then $\V$ is $\Lambda_{\V|\d^0_\V}$-Varsovian.

We finish this section by introducing the universes that are the companions of hod-like pairs.
\begin{definition}\label{companions} We say that $M$ and $(\P, \Sigma)$ are \textbf{companions} if the following conditions hold.
\begin{enumerate}
\item Letting $\a={\sf{Ord}}\cap M$, $M=L_\a(\bR)$, and for some sentence $\phi$, $\a$ is the least $\b$ such that $L_\b(\bR)\models ``{\sf{ZF-Replacement}}+\phi."$
\item $M$ is the derived model of $(\P, \Sigma)$.
\item There is a sjs $(B_i: i<\omega)\subseteq \powerset(\bR)\cap L_\a(\bR)$ such that for each $i<\omega$, $B_i$ is ordinal definable in $M$, and $\Sigma$ is the unique $\omega_1+1$-iteration strategy $\Lambda$ such that for every $i\in \omega$, $\Lambda$ respects $B_i$.
\end{enumerate}
We say $M$ has a companion if there is a pair $(\P, \Sigma)$ such that $M$ and $(\P, \Sigma)$ are companions.
\end{definition}
The next theorem, the main result on companions, can be proved using the methods of \cite{MSCR}.
\begin{theorem}\label{companions thm} Assume $V=L(\bR)+{\sf{AD}}$. Suppose $\a$ is such that for some sentence $\phi$, $\a$ is the least $\b$ such that $L_\b(\bR)\models ``{\sf{ZF-Replacement}}+\phi."$ Then $L_\a(\bR)$ has a companion.
\end{theorem}

\begin{remark} In \rthm{companions thm}, the desired $(\P, \Sigma)$ is built using hod pair constructions as in \cite{MSCR}. Clause 3 of \rdef{companions} can be achieved by fixing a sjs system for $L_\a(\bR)$, which can be done by the results of \cite{St08b}, \cite{Jackson}, \cite{HMMSC}, \cite{Trev1} and \cite{Trev2}. Clause 2 of \rdef{hod like pairs} is more or less automatic and has been treated extensively in the literature, e.g. \cite{NegRes}. 
\end{remark}

\subsection{\textbf{On HOD analysis.}}
We exposit the HOD analysis of $L_\a(\bR)$. Fix some ordinal $\a$ such that $L_\a(\bR)\models {\sf{ZF-Replacement}}$. We allow $\a={\sf{Ord}}$.

Recall from \cite[Theorem 6.1]{HODCoreModel} that assuming $V=L(\bR)$, ${\sf{HOD}}=L[\M_\infty^+, \Lambda]$. Hence ${\sf{HOD}}^{L_\a(\bR)}=L[\M_\infty^{ +, \a}, \Lambda^\a]$, where $\M_\infty^{+, \a}$ is a hod-like premouse with exactly $\omega$-Woodin cardinals, ${\sf{ORD}}\cap \M_\infty^{+, \a}=\a$ and $\M^{+, \a}_\infty$ is definable in $L_\a(\bR)$ via a direct limit construction (see \cite{HODCoreModel}). We set $\M^\a=\M_{\infty}^{+, \a}$. $\Lambda^\a$ is a partial iteration strategy for $\M^\a$ that acts on iteration trees which are in $\M^\a|\d^\omega_{\M^\a}$ and are based on $\M^\a|\d^0_{\M^\a}$. $\Lambda^\a$ induces a branch $b^\a$ for $\T^\a=_{def}\T_{\M^\a}$ that is friendly to $\M^\a$. In fact, $\Lambda^\a$ and $b^\a$ are definable from each other (see \cite{HODCoreModel}). We set $\V'_\a=\V'(\M^\a, b^\a)$ and $\V_\a=\V(\M^\a, b^\a)$.\footnote{Our arguments below will show that $(\M^\a, b^\a)$ is hod-like. See \rprop{complete rep of hod}.} We then have that ${\sf{HOD}}^{L_\a(\bR)}=L_\a[\V_\a]$.\footnote{This is because if $A$ is the Vopenka algebra of $L_\a(\bR)$ and $A'$ is the Vopenka algebra of $D(\M^\a, \d^\omega_{\M^\a})$, then we have that $\xi\in A$ if and only if $\pi^{\T^\a}_{b^\a}(\xi)\in A'$. Thus $A$ is definable in $L[\V'_\a]$. The same procedure also defines $A$ in $L[\V_\a]$.} In fact more is true. Below and elsewhere, when studying objects like $\V=_{def}\V(\P, b)$, we will let $\X^\V=\P$, $\T^\V=\T_\P$ and $b^\V=b$.

\begin{proposition}\label{complete rep of hod} Suppose $L_\a(\bR)$ and $(\P, \Sigma)$ are companions. There is a Varsovian pair $(\R, \Psi)$ such that $\V_\a$ is the direct limit of all complete $\Psi$-iterates $\Q$ of $\R$ such that $\T_{\R, \Q}$ is based on $\R|\d^0_\R$.
\end{proposition}
\begin{proof} First find some coarse tuple $(R_0, \Psi_0, H, \a')$ that absorbs $\a$, as in \rthm{woodins thm}.\footnote{Because we are working in $L(\bR)$, we in fact can assume that $(R_0, \Psi_0)=\emptyset$.} Next, let $x\in \dom(H)$ be such that, letting $(\Q, \Lambda)=H(x)$, $\Sigma$ is Suslin, co-Suslin captured by $(\Q, \Lambda)$. 

Following \cite{HMMSC} and \cite{SteelCom}, we can find a complete $\Sigma$-iterate $\W$ that is built using the fully backgrounded construction of $\Q$. Let $\Phi=\Sigma_{\W}$. It follows from the results of \cite{HMMSC} and \cite{SteelCom} that $\Phi$ is the strategy of $\W$ induced by $\Lambda$. It follows from \cite{HMMSC} and \cite{SteelCom} that $\W$ is hod-like, and if $b=\Phi(\T_\W)$, then $(\W, b)$ is hod-like. Set $\R'=\V'(\W, b)$, and let $\Psi'$ be the strategy of $\R'$. 

It follows that whenever $\R''$ is a complete $\Psi'$-iterate of $\R'$, $b^{\R''}$ is according to $\Psi'_{\X^{\R''}}=\Sigma_{\X^{\R''}}$. Because $\M^\a=\M_\infty(\P, \Sigma)$, we have that $\V'_\a=\M_\infty(\R', \Psi')$. Hence, the core of $\V'_\a$ is defined, and letting $\R=\V(\W, b)$ and $\Psi$ be the strategy of $\R$ induced by $\Psi'$, $\V_\a$ is the direct limit of all complete $\Psi$-iterates $\Q$ of $\R$ such that $\T_{\R, \Q}$ is based on $\R|\d^0_\R$.\footnote{This argument also shows that the projectum of $\R$ is $\d^0_\R$.}
\end{proof}

\subsection{\textbf{\rthm{main theorem lr} cannot be improved}}
\begin{theorem}\label{lr optimal} Assume $V=L(\bR)+{\sf{AD}}$. Let $\mH$ be the premouse representation of $V_\Theta^{\sf{HOD}}$, and suppose $j: \mH\rightarrow \mH$ is elementary. Then $j=id$.
\end{theorem}
\begin{proof}
It is a well-known theorem of Woodin that in $L(\bR)$, $\mH=L[A]$, where $A\subseteq \Theta$ is the set of ordinals coding the Vopenka algebra in some natural way (see \cite{ADPlusBook}, \cite{Tr14} or \cite{HODCoreModel}). We now want to show that $j$ can be extended to $j^+: L[A]\rightarrow L[A]$. Because $j^+$ cannot be added to HOD by a set forcing, we can then use the proof of \rprop{gabe} to show that in fact $A^\#$ exists. We then have an embedding $k: {\sf{HOD}}\rightarrow {\sf{HOD}}$ with $\cp(k)>\Theta$, which then induces $k^+: L(\bR)\rightarrow L(\bR)$. This is because $L(\bR)$ is a symmetric extension of ${\sf{HOD}}$ by a poset of size $\Theta$ (see \cite{ADPlusBook}). Below $\Theta^\gg=\Theta^{L_\gg(\bR)}$.

\begin{lemma}\label{extension} Let $E$ be the extender derived from $j$.\footnote{More precisely, $(a, B) \in E$ if and only if $a\in \Theta^{<\omega}$, $B\in \mH$ and $a\in j(B)$.} Then $Ult(L[\V_{{\sf{Ord}}}], E)$ is well founded and is equal to $L[\V_{{\sf{Ord}}}]$. 
\end{lemma}
\begin{proof}
 Let $\phi$ be the sentence we are trying to prove. Towards a contradiction, assume $\phi$ is false. Let $\a$ be the least $\gg$ such that $L_\gg(\bR)\models ``{\sf{ZF-Replacement}}+\neg \phi"$ and $\gg$ is a limit of ordinals $\b$ such that $L_\b(\bR)\models ``{\sf{ZF-Replacement}}+\neg \phi."$ Let $(\R, \Psi)$ be as in \rprop{complete rep of hod} applied to $\a$. 

We now reflect inside $L_\a(\bR)$ and find\\\\
(1.1) $\b\in (\Theta^\a, \a)$ such that $L_\b(\bR)\models ``{\sf{ZF-Replacement}}+\neg \phi,"$ and\\
(1.2) $\gg<\Theta^\a$ and $\sigma: L_\gg(\bR)\rightarrow L_\b(\bR)$ such that $j\in \rge(\sigma)$ and $\Theta^\gg$ is a regular cardinal.\footnote{\cite{KoelWoodin} shows that $\Theta^\a$ is a Mahlo cardinal in $L_\a(\bR)$. Indeed, if $C\in L_\a(\bR)$ is a club subset of $\Theta^\a$ and $\d<\Theta^\a$ is the least such that $L_\d(\bR, C)\prec_1^\bR L_{\Theta^\a}(\bR, C)$, then $\d$ is a measurable cardinal as shown in \cite{KoelWoodin}, and clearly, $\d\in C$. To get such an elementary embedding, notice that we can find, using ${\sf{DC}}$ and the usual argument for building Skolem hulls, a countable $X\prec L_\b(\bR)$ such that for any $\zeta<\Theta^\a$, letting $\Delta_\zeta=\{B\subseteq \bR: w(B)<\zeta\}$ and  $X[\Delta_\zeta]=\{f(u): u\in \Delta_\zeta \wedge f\in X\}$, $X[\Delta_\zeta]\prec_1^\bR L_\b(\bR)$.}\\\\
We thus have that, letting $F=\sigma^{-1}(E)$,\\\\
(2.1) $L_\gg(\bR)\models ``Ult(L[\V_\gg], F)$ is ill-founded or $L[\V_\gg]\not = Ult(L[\V_\gg], F)."$\\\\
We now establish a sequence of claims leading to the proof of \rlem{extension}.

Because $\sigma\rest \V'_\gg:\V'_\gg\rightarrow \V'_\a$, we have that $\V'_\gg$ is iterable via the $\sigma$-pullback of $\Psi_{\V'_\a}$. Let $\Phi$ be the $\sigma$-pullback of $\Psi_{\V'_\a}$ and $\Phi'$ be the fragment of $\Phi$ that acts on iteration trees that are above $\Theta^\gg$. Notice that\\\\
(3.1) $\Phi_{\V'_\gg|\Theta^\gg}=\Psi_{\M^\a|\Theta^\gg}$.\\
\begin{claim}\label{1} $\V'_\gg\in \M^\a$, $\Phi'\rest \M^\a|\Theta^\a \in \M^\a$, and $\Phi'\rest \M^\a|\Theta^\a$ has a $\Theta^\a+1$-extension in $\M^\a$. 
\end{claim}
\begin{proof} Because $\V'_\gg\in {\sf{HOD}}^{L_\gg(\bR)}$, we have that $\V'_\gg\in \M^\a$. We show that $\Phi'\rest \M^\a|\Theta^\a \in \M^\a$. The proof will also show the third clause.

Suppose $\T$ is a normal tree on $\V'_\gg$ that is above $\Theta^\gg$. Then $\T$ naturally splits into a stack of $\omega$-many normal iteration trees such that the $i$th normal iteration tree in the stack is based on the $i$th window (where by \emph{window} we mean a maximal interval $(\xi, \xi')$ that contains no Woodin cardinals). In light of this observation, it is enough to show that for each $i<\omega$, if $\T\in \M^\a|\Theta^\a$ is a normal iteration tree according to $\Phi'$ with last model $\S$ such that $\T$ is based on $\V'_\gg|\d^i_{\V'_\gg}$ and the main branch of $\T$ doesn't drop, then the fragment of $\Phi'_{\S}\rest \M^\a|\Theta^\a$ that acts on stacks that are above $\d^i_\S$ and below $\d^{i+1}_\S$ is in $\M^\a$. 

We prove this assuming $\T=\emptyset$ to simplify the notation, the general proof being only notationally more complex. Thus, set $\S=\V'_\gg$, and notice that $\Phi\rest \M^\a|\d^\omega_{\M^\a}\in \M^\a$. We thus prove the above claim for $i=1$. More precisely, we show that if $\Lambda$ is the fragment of $\Phi'\rest \M^\a|\Theta^\a$ that acts on normal iteration trees that are based on the interval $(\d^0_\S, \d^1_{\S})$, then $\Lambda\in \M^\a$.

Suppose then $\U\in \M^\a|\Theta^\a$ is a normal iteration tree on $\S$ based on the interval $(\d^0_\S, \d^1_{\S})$ such that $\U$ has a limit length and is according to $\Lambda$. It is enough to show that if $c=\Lambda(\U)$, then $c$ is uniformly definable over $\M^\a$ from $\U$ and $\Theta^\gg$.

We have two cases. Suppose first that either $c$ has a drop or $\pi^\U_c(\d^1_\S)>\d(\U)$. Either way, $\Q(c, \U)$ is defined, and whenever $\tau: \W\rightarrow \Q(c, \U)$ is such that $\tau\in L_\a(\bR)$ and $\W$ is countable, $\W$ has a $\omega_1+1$-iteration strategy in $L_\a(\bR)$. It now follows that $\Q(c, \U)\in \M^\a$ and is uniformly definable from $\U$ and $\Theta^\gg$.\footnote{For example, it appears in the fully backgrounded construction of $\M^\a|\Theta^\a$ done over $\cop(\U)$.} 

Suppose $c$ doesn't have a drop and $\pi^\U_c(\d^1_\S)=\d(\U)$. Let $\S'=\M^\U_c$. Let $\Y\in \S$ be the normal tree on $\S|\d^0_\S(=\V'_\gg|\Theta^\gg)$ such that $\cop(\Y)$ is the output of the fully backgrounded construction of $\S|\d^1_\S$ using extenders whose critical points are above $\d^0_\S$. Notice that $\Y'=_{def}\pi^\U_c(\Y)$ only depends on $\cop(\U)$, and also if $\Phi(\Y)=d$ and $\Phi(\Y')=d'$, then $\pi^{\Y'}_{d'}=\pi^\U_c\rest \d^1_\S \circ \pi^\Y_d$.\footnote{This uses the fact that $\pi^\U_c(b^\S)=\Phi(\pi^\U_c(\T^\S))$.} Since $\Phi\rest \M^\a|\Theta^\a\in \M^\a$, we have that $(\Y', d')\in \M^\a$ is uniformly definable from $\U$. We now have that $c$ is the unique branch of $\U$ such that $\rge(\pi^{\Y'}_{d'})\subseteq \rge(\pi^\U_c)$. Hence, $c\in \M^\a$ and is uniformly definable in $\M^\a$ from $\U$ and $\Theta^\gg$.
\end{proof}

The next claim follows immediately from \rcl{1}. Because $\V'_\gg\in \M^\a$, and because $Ult(\M^\a, F)$ is well-founded (as $F$ is derived from $j$), $Ult(L_\gg[\V'_\gg], F)$ is well-founded. We thus need to show that $Ult(L_\gg[\V'_\gg], F)=L_\gg[\V'_\gg]$. Notice that $L_\gg[\V_\gg]=L_\gg[\V'_\gg]={\sf{HOD}}^{L_\gg(\bR)}$, so it is enough to show that $Ult(L_\gg[\V_\gg], F)=L_\gg[\V_\gg]$.

For $\xi<\Theta^\a$, let  $\K_\xi\insegeq \M^\a$ be the longest initial segment $\X$ of $\M^\a$ such that $\X\models ``\xi$ is a Woodin cardinal." Let $\Lambda_\xi\in \M^\a$ be the unique $(\xi^+)^{\M^\a}+1$-strategy of $\K_\xi$. 

\begin{claim}\label{2} In $\M^\a|\Theta^\a$, $\V_\gg$ is the unique $\Theta^\gg$-sound $\Lambda_{\Theta^\gg}$-Varsovian model.\footnote{See \rdef{iteration of v-models}.}
\end{claim}

\begin{claim}\label{3} $\pi_F^{L_\gg[\V_\gg]}(\V_{\gg})=\V_\gg$. 
\end{claim}
\begin{proof} Just like with $\V'_\gg$, we have that $\V_\gg\in \M^\a$. Let $k: Ult(\M^\a|\Theta^\a, F)\rightarrow  \M^\a|\Theta^\a$ be such that $j=k\circ \pi_F^{\M^\a|\Theta^\a}$. Notice that $\cp(k)\geq \Theta^\gg$, and that we have $\tau: \pi_F^{L_\gg[\V_\gg]}(\V_\gg)\rightarrow \pi_F^{\M^\a|\Theta^\a}(\V_\gg)$ with $\cp(\tau)\geq \Theta^\gg$ such that $\pi_F^{\M^\a|\Theta^\a}\rest \V_\gg=\tau\circ \pi_F^{L_\gg[\V_\gg]}\rest \V_\gg$. We thus have that
\begin{center}
$j\rest \V_\gg=k\circ \tau\circ \pi_F^{L_\gg[\V_\gg]}\rest \V_\gg$.
\end{center}
It is enough to show that $\pi_F^{L_\gg[\V_\gg]}(\V_\gg)\in \M^\a$ and $\M^\a\models ``\pi_F^{L_\gg[\V_\gg]}(\V_\gg)$ is the $\Lambda_{\Theta^\gg}$-Varsovian model over $\M^\a|\Theta^\gg$". Notice that\\\\
(4.1) $\M^\a|\Theta^\a\models ``j(\V_\gg)$ is the unique $j(\Theta^\gg)$-sound $\Lambda_{j(\Theta^\gg)}$-Varsovian model."\\
(4.2) $\Lambda_{\Theta^\gg}$ is the $k\circ \tau$-pullback of $\Lambda_{j(\Theta^\gg)}$.\\\\
It then follows from (4.1) and (4.2), and from the fact that $\pi_F^{L_\gg[\V_\gg]}(\Theta^\gg)=\Theta^\gg$, that $\pi_F^{L_\gg[\V_\gg]}(\V_\gg)$ is the unique $\Theta^\gg$-sound $\Lambda_{\Theta^\gg}$-Varsovian model. Hence, it follows from \rcl{2} that $\pi_F^{L_\gg[\V_\gg]}(\V_\gg)=\V_\gg$.
\end{proof}
This finishes the proof of \rlem{extension}.
\end{proof}
Since $L[\V_{{\sf{Ord}}}]={\sf{HOD}}^{L(\bR)}$, \rlem{extension} implies \rthm{lr optimal}. 
\end{proof}

The proof of \rthm{lr optimal}, \cite{HMMSC} and \cite{HNMAD} can be used to show the following.

\begin{theorem}\label{thetareg optimal} Suppose $V$ is the minimal model of ${\sf{AD}}_{\bR}+``\Theta$ is a regular cardinal", $\mH$ is the hod premouse representation of $V_\Theta^{\sf{HOD}}$ and $j:\mH\rightarrow \mH$ is an elementary embedding. Then $j=id$.
\end{theorem}

\bibliographystyle{plain}
\bibliography{main}
\end{document}